\documentclass[11pt]{amsart}
\usepackage[letterpaper,left=75pt,right=75pt]{geometry}
\allowdisplaybreaks[0]
\usepackage{float}
\usepackage{graphicx}

\input{preamble.sty} 

\date{\today}

\begin{document}
\begin{abstract}

Chatawate (Flame) Ruethaimetapat was a passionate, enthusiastic, and wonderful person who passed away in August of 2024. At the time of their passing they were working towards their PhD, specializing in geometric group theory. Flame was just as excited about learning new mathematics as they were about sharing it with everyone else, so it's no surprise that they spent a lot of time thinking about how to write down expository proofs of classical theorems that would be accessible for first year students. In particular, they sought a simple, elementary proof of the fact that any finitely generated group quasi-isometric to the integers is virtually $\mathbb{Z}$. In the spirit of this endeavor and in loving memory of Flame, we present such a proof here.
\end{abstract}

\title{Quasi-isometric rigidity of the integers: an elementary primer}
\input{authors.sty}
\maketitle
\section*{Introduction}

In this article, we present a completely elementary proof of the following well-known theorem in geometric group theory.

\begin{maintheorem} \label{thm:main} Let $G$ be a finitely generated group that is quasi-isometric to $\mathbb{R}$; then $G$ is virtually $\mathbb{Z}$. 
\end{maintheorem}

There are many ways to conclude this result using modern machinery. Our goal here is to offer a proof that does not rely on any other results. It is a proof written ``from scratch" so to speak. It is meant to provide students entering the subject with a solid understanding of the notion of quasi-isometry - i.e. what it does and does not give you. At the end of the proof, we offer a few ways to see this result using well-known concepts and results from geometric group theory that one could explore from here. 

\subsection*{Flame}

We were brought to this project by our various connections to Chatawate (Flame) Ruethaimetapat, a beloved student and teacher of mathematics who passed away in the Summer of 2024. While studying for their candidacy exam in geometric group theory, Flame was particularly enthralled by Theorem \ref{thm:main} and became very interested in devising a proof that could be easily explained to a student with no background whatsoever in geometric group theory. The argument we present below was inspired by various pictures and ideas with which Flame was playing while working towards this goal. 

Flame's passion for exposition and for sharing the joy of experiencing mathematical insight was facilitated through their art work. At Haverford College, they double-majored in mathematics and art, and their senior thesis was recognized by the mathematics department for its beautiful and clear hand-drawn figures. While none of us possess even remotely close to Flame's level of artistic skill, in the spirit of their memory we chose to hand-draw our own figures. We highlight some of Flame's mathematical drawings in Appendix C, but we begin the paper with the drawing below which Flame created while thinking about Theorem \ref{thm:main}.

\begin{center}
\includegraphics[height=4cm]{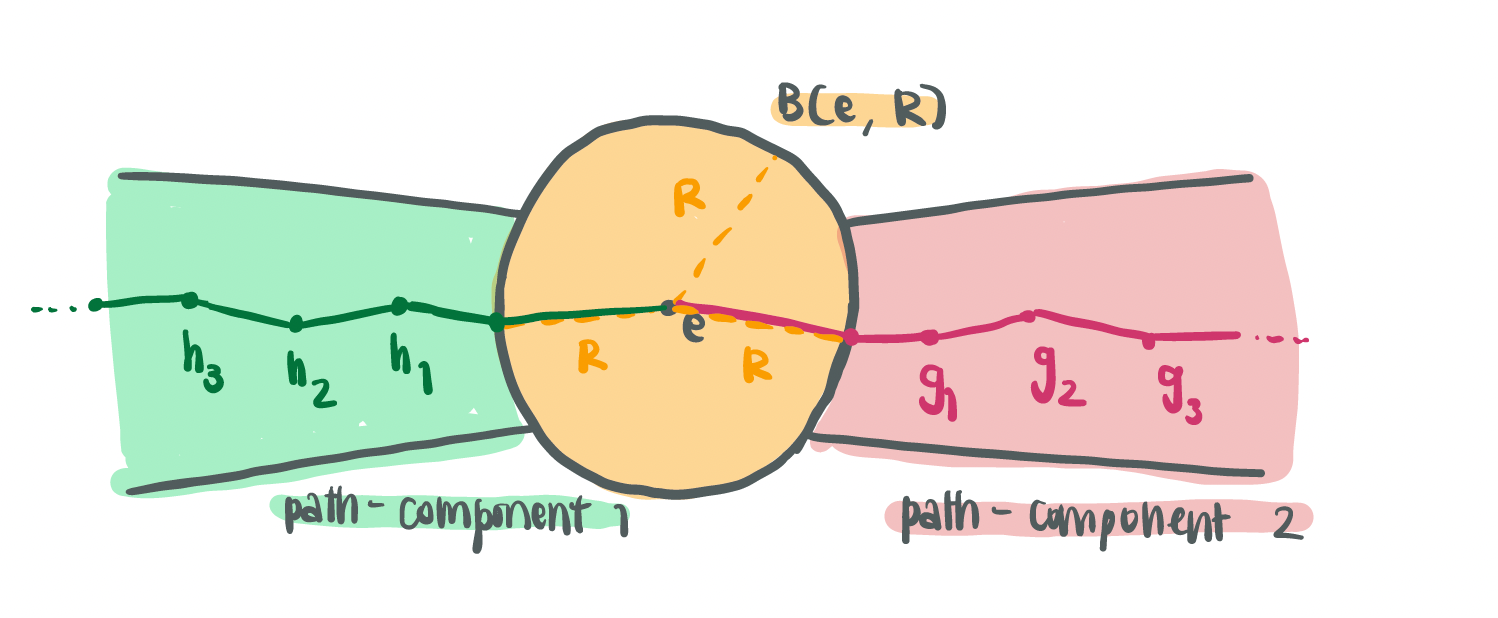}
\end{center}

\subsection*{What you'll need to read this and what you'll find in it (and what you won't)} To understand the sketch of Theorem \ref{thm:main}, a reader will only need a comfortable understanding of group theory at the level of an undergraduate course. In particular, we rely on the notions of a generating set for a group and a group action on a set. The reader will be introduced to the concepts of \textbf{quasi-isometries} between groups and spaces and of course to the central idea in geometric group theory of treating the (a) Cayley graph of a group as a space whose geometry encodes information about the algebraic properties of the group. 

We stress that this is \textit{not} the place for an ambitious undergraduate to find a comprehensive introduction to all of modern geometric group theory. We refer the reader to other excellent resources (for instance \cite{clay2017office}) if this is what they are after. It is also not the place for a more extensive graduate-level review of the field (one could look to \cite{bridson_metric_1999}, \cite{farb2011primer}, \cite{meier_groups}, or \cite{loh_geometric_group_theory} for something like this). Rather, our goal is to present one of the simplest theorems in geometric group theory that is also deeply interesting, and to do it in a way that introduces an uninitiated reader to some of the basic and central concepts of the field. 

We hope, more than anything, that this piece encourages young graduate students to share the ideas with each other and with younger students and  undergraduates in their orbit. This is what Flame would have wanted.

\subsection*{Outline} In the next section we will introduce the basic players and discuss a quasi-action on $\mathbb{R}$. Then we will describe how this in turn gives rise to an action on the two ``ends" of the real line. Finally, we will hunt for (and find) a copy of $\mathbb{Z}$ in our group and we will verify that the entire group is quasi-isometric to this. 

We include several appendices: in Appendix A, we discuss some ideas for alternative proofs. In Appendix B, we summarize some ideas (largely inspired by an old MathOverflow post of Thurston \cite{MathOverflowPost}) for proving a similar conclusion but only under the weakened hypothesis that $G$ has linear growth. Finally, in Appendix C, we include some of Flame's original hand-drawn mathematical figures. 

\section*{Quasi-Isometry}
We begin by introducing the notions of a \textbf{Cayley Graph} and a \textbf{Quasi-Isometry}. A reader familiar with these definitions can move on to the proof outline in the next section.

In a first course in group theory, the notion of a group is often geometrically motivated. In the world of finite groups, the dihedral groups can be given as rigid symmetries of a regular polygon, while infinite groups such as the general linear groups can be given as a special group of symmetries of $\R^n$. These are often given as first examples, since being able to `see' the group act on a space can help us understand its structure more deeply.

But when we axiomatically define groups, we risk losing some of that geometric structure. For instance, what is the right geometric interpretation for the following group?
\[\langle a,b\:|\: bab^{-1}=a^2\rangle\]

At a glance, this group looks fairly uncomplicated. It is given by just two generators and one relation. But even so, it's hard to see it geometrically just by looking at the symbols in its presentation. See Appendix C for a picture of part of this Cayley graph.

The first definition we present is a geometric group theorist's most fundamental tool for tackling this issue. How can we look at any group and find some geometric object to represent it?

\begin{definition}\label{def:CayleyGraph}
    For a group $G$ with some generating set $S$, the \textbf{Cayley Graph} $\Gamma_{S}$ is built as follows:
    \begin{enumerate}
        \item Create a vertex for every element of the group.
        \item Connect vertex $g$ to vertex $g'$ by an edge exactly when there is some $s \in S$ so that either $g = sg'$ or $g'= sg$. Label that edge with the appropriate generator and orient it so that left multiplying the initial vertex by the generator results in the final vertex.
    \end{enumerate}
    The Cayley Graph is given a metric $d_{\Gamma_{S}}$ by making the length of every edge exactly $1$.
\end{definition}

For a single group, there are in general many different Cayley graphs. For instance, Figure \ref{fig:CayleyGraphs} shows two Cayley graphs for the same group $\Z$. The top graph connects two vertices when they differ by $1$, the bottom graph connects two vertices when they differ by $2$ or $3$.

\begin{figure}[b]
    \centering
    \includegraphics[width=0.5\linewidth]{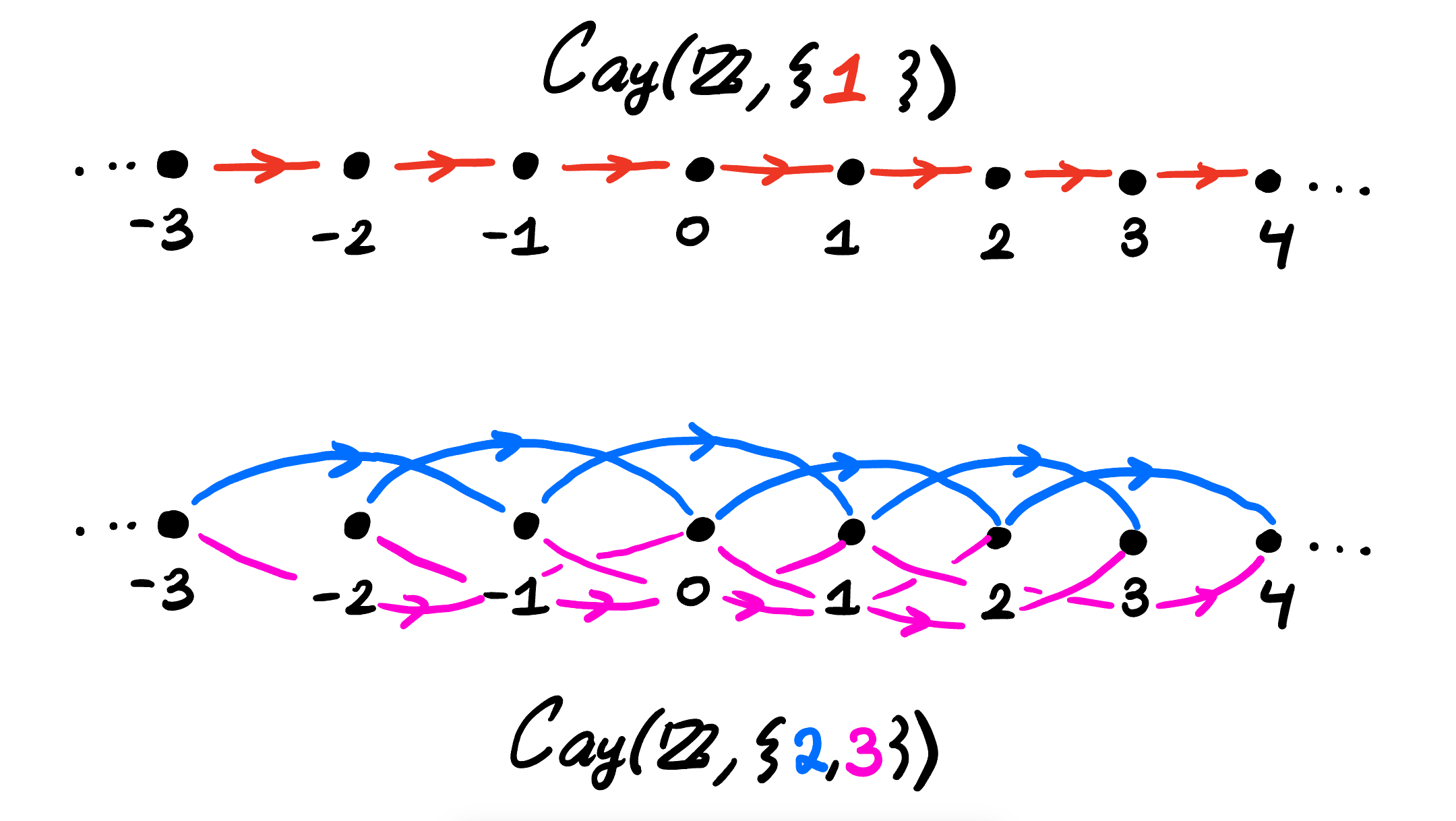}
    \caption{Two Cayley Graphs for $\Z$}
    \label{fig:CayleyGraphs}
\end{figure}

A key point we would like to emphasize is that $G$ acts on $\Gamma_{S}$ via left multiplication, and this action is by isometries with respect to the metric $d_{\Gamma}$: 
\[ d_{\Gamma_{S}}(g, g') = d_{\Gamma_{S}}(hg, hg'), \forall g,g',h \in G. \]

But even though a group doesn't have a unique Cayley graph, if we only look at finite generating sets they all share something in common. We formalize this below:

\begin{definition}
    A map $f:X\to Y$ between metric spaces is called a $(\lambda,\epsilon)$ quasi-isometric embedding if there exists $\lambda\geq 1$ and $\epsilon\geq 0$ so that for all $a,b \in X$:
    \[\frac{1}{\lambda}d_X(a,b)-\epsilon\leq d_Y(f(a),f(b))\leq \lambda d_X(a,b)+\epsilon.\]
\end{definition}

In other words, a quasi-isometric embedding is a map between metric spaces that can only alter distances a `small' amount. The constants $\lambda$ and $\epsilon$ encode numerically how much error is allowed.

\begin{figure}[b]
\centering
    \begin{subfigure}{0.45\textwidth}            
            \includegraphics[width=\textwidth]{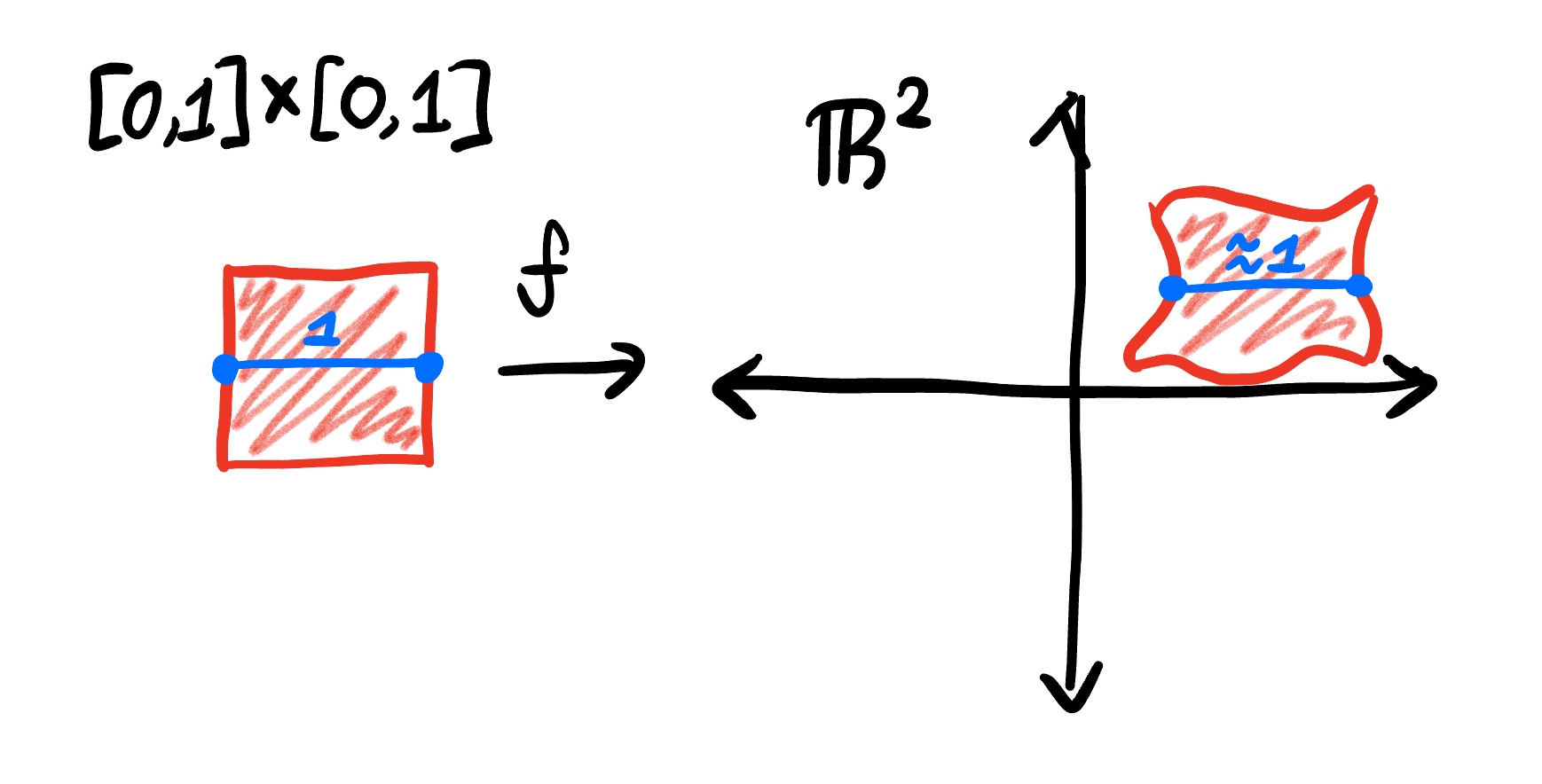}
            \caption{A Quasi-Isometric Embedding from the Square to the Plane}
            \label{fig:QIEmbedding}
    \end{subfigure}%
     \quad 
    \begin{subfigure}{0.45\textwidth}
            \centering
            \includegraphics[width=\textwidth]{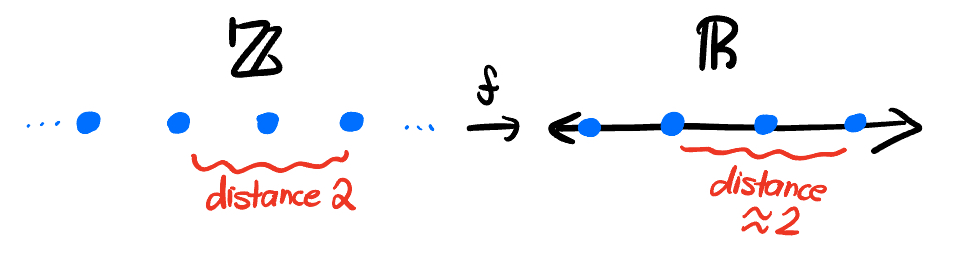}
            \caption{A Sketch of a Quasi-Isometry from $\Z$ to $\R$.}
            \label{fig:QI}
    \end{subfigure}
    \caption{Quasi-Isometry}\label{}
\end{figure}

We also want some notion of when a map is almost surjective, which is given by the following:
\begin{definition}
    A subset $Y \subseteq X$ is called \textbf{quasi-dense} if there is a $k\geq 0$ so that every point in $X$ is within distance $k$ of a point in $Y$.
\end{definition}

If a quasi-isometric embedding has quasi-dense image, we call it a \textbf{quasi-isometry}.

\begin{remark}
From the point of view of topology, quasi-isometries do not have the same level of rigidity as isometries. For example, quasi-isometries are not necessarily homeomorphisms, and they are often non-continuous and non-invertible. But they do still act like an equivalence between spaces in the following sense:
    \begin{enumerate}
        \item The composition of two quasi-isometries is a quasi-isometry.
        \item For any quasi-isometry $f:X\to Y$, there is a quasi-isometry $f^{-1}:Y\to X$ called a \textbf{quasi-inverse}. This map is not truly a setwise inverse, but has the property that $f^{-1}\circ f:X\to X$ and $f\circ f^{-1}:Y \to Y$ moves points at most distance $k$ for some fixed $k>0$.
    \end{enumerate} 
\end{remark}

This notion of quasi-isometry turns out to be the right one for thinking about finitely generated groups, since if $\Gamma$ and $\Gamma'$ are two Cayley graphs with finite generating sets for a group $G$, then they are quasi-isometric to one another.\footnote{The quasi-isometry between $\Gamma$ and $\Gamma'$ is the map which takes any vertex in $\Gamma$ to the vertex in $\Gamma'$ representing the same element, and takes an edge in $\Gamma$ to the shortest path in $\Gamma'$ between its initial and final vertices.}

With this information, can now treat a finitely generated group as a single geometric object. Thus, when we say ``$G$ is quasi-isometric to $\R$" in the theorem statement, we mean ``some (and thus any) Cayley graph for $G$ with finite generating set has a quasi-isometry to $\R$". 

\section*{The set-up} \label{set-up}
Fix a finite generating set for $G$ and let $\Gamma$ denote the Cayley graph for $G$ with respect to this generating set. We present the proof in $4$ steps: 

\begin{itemize}
\item Step 1: Use the quasi-isometry to $\mathbb{R}$ to turn the group action $\bullet:G \times \Gamma \to \Gamma$ into a map $*:G \times \mathbb{R} \to \mathbb{R}$ that coarsely captures the information of the group action.

\vspace{2 mm}

\item Step 2: Use the map $*$ to construct a group action on $\{\pm \infty\}$, the two ends of the real line. This will allow us to isolate the elements of $g$ that coarsely act as ``translations" on the line.

\vspace{2 mm}

\item Step 3: Show that for some $g \in G$, the orbit $\{g^n*0\}_{n \in \mathbb{Z}}$ is quasi-dense in $\mathbb{R}$. Then use the relationship between $*$ and $\bullet$ to show that $\{g^n \bullet e\}_{n \in\mathbb{Z}}$ is quasi-dense in $\Gamma$.  This $g$ generates the $\mathbb Z$ of finite index in $G$.

\end{itemize}

\section*{Step 1: Passing from $\Gamma$ to $\R$} 
Let us begin by fixing some notation that we will use throughout this proof.
\begin{notation}
Let $\Gamma$ be a Cayley graph for $G$ with respect to some fixed generating set. 
    \begin{enumerate}[(A)]
    \item Let $\varphi:\Gamma \to \mathbb{R}$ be our $(\lambda,\epsilon)$ quasi-isometry. By post-composing by a translation if necessary, we can assume without loss of generality that $\varphi(e)=0$.
    \item Let $\varphi^{-1}$ be the quasi-inverse of $\varphi$. We can assume that $\varphi^{-1}(0)=e$ and $\varphi^{-1}$ is also a $(\lambda,\epsilon)$ quasi-isometry. 
    \item By the definition of the quasi-inverse, we can find a $k\geq 0$ so that that $d(x,\varphi^{-1}\circ\varphi(x))\leq k$ for all $x$. For notational convenience, we can assume without loss of generality that $k=\epsilon$.
    \item Let $\bullet:G \times \Gamma \to \Gamma$ be the action of $G$ on its Cayley graph by isometries. 
\end{enumerate}
\end{notation}

Now, our overall goal is to find a copy of $\Z \leq G$ so that $\Z \bullet \{e\}$ lies as a quasi-dense subset of $\Gamma$. But we quickly run into an issue, as nothing is given to us about the group structure of $G$. How do we even know that there is a copy of $\Z$ in $G$ at all? Luckily, we do have a quasi-isometry to the space $\R$, which is (quasi-isometric to) the Cayley graph of $\Z$ with its standard generating set, $\left\{ 1 \right\}$. Maybe if we turn the group action $\bullet$ on $\Gamma$ into a nice enough map on $\R$, we can push forward properties of $\Gamma$ and $\R$ to analogous properties of the corresponding groups, $G$ and $\Z$. 

We do this using the map $*$, which in turn calls on the quasi-isometry $\varphi$ to translate between $\mathbb{R}$ and $\Gamma$.

\begin{definition}
The map $*$, as illustrated in Figure \ref{fig:Star_Action}, is given by:
\begin{align*}
    *:G \times \mathbb{R} &\to \mathbb{R}\\
    (g,x)&\mapsto \varphi\circ g\bullet(\varphi^{-1}(x))
\end{align*}
\end{definition}

\begin{figure}[h!]
    \centering
    \includegraphics[width=0.5\linewidth]{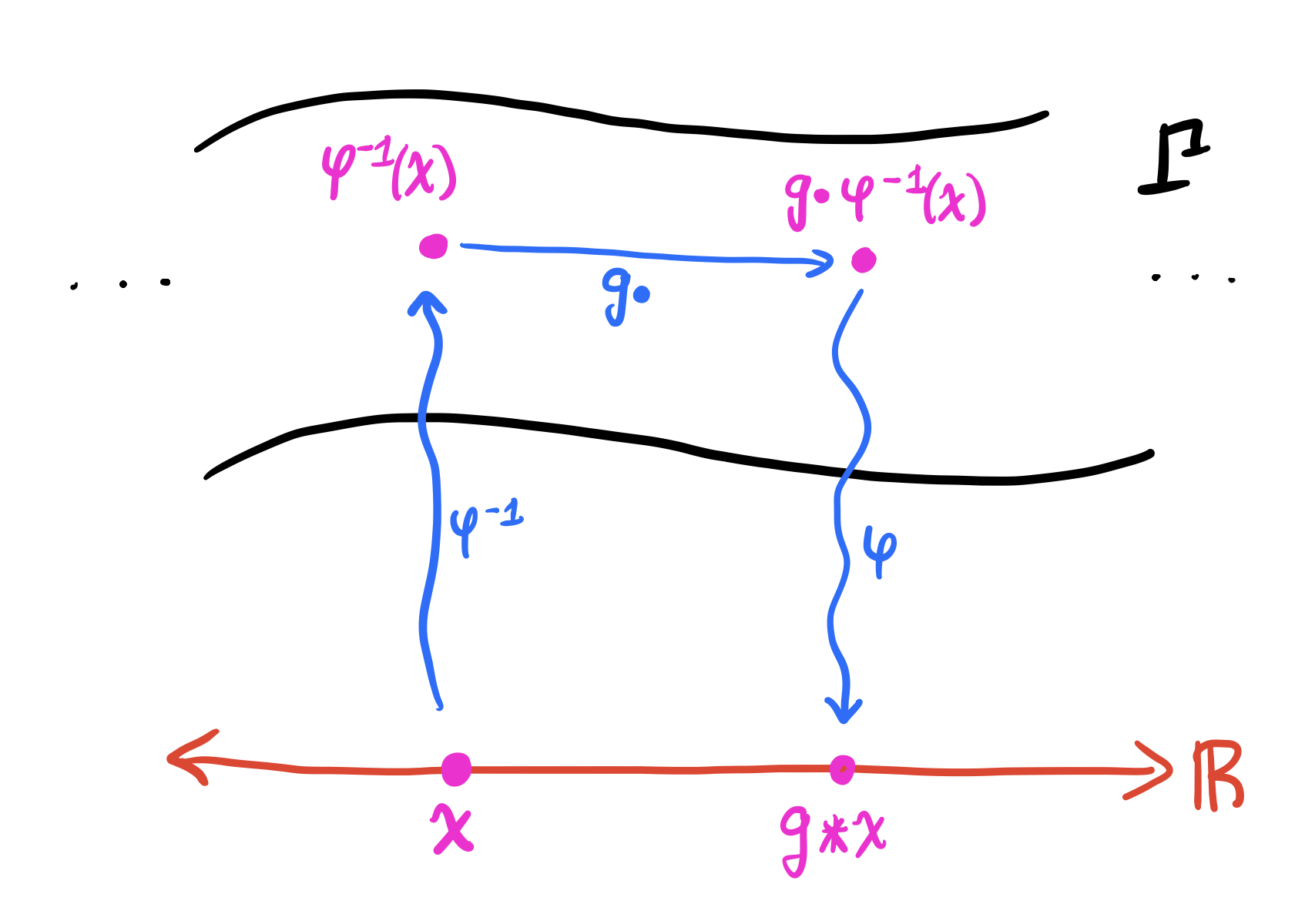}
    \caption{The map $*$ applied to an element of $\R$}
    \label{fig:Star_Action}
\end{figure}

Using the fact that $*$ is a composition of isometries and quasi-isometries, we have a great deal of control over where it sends elements.

\begin{lemma} \label{lem:four properties}
    The map $*$ has the following properties:
    \begin{enumerate}[(A)]
    \item This map behaves like a group action, up to a small error.
    \begin{enumerate}[(1)]
        \item $d_\mathbb{R}((gh)*x,g*(h*x))\leq \epsilon+\lambda\epsilon$.
        \item $d_\mathbb{R}(e*x,x)\leq \epsilon$.
    \end{enumerate}
    \item For each $g \in G$, the self-map of $\R$ given by $*(g, \cdot)$ is a quasi-isometry:
    \begin{enumerate}
        \item[(3)] $\frac{1}{\lambda^2}d_{\mathbb{R}}(x,y)-\frac{\epsilon}{\lambda}-\epsilon \leq d_\mathbb{R}(g*x,g*y)\leq \lambda^2 d_\mathbb{R}(x,y)+\lambda\epsilon+\epsilon$.
    \end{enumerate}
    \item The vertex $g \bullet e$ of $\Gamma$ lies nearby to the vertex $\varphi^{-1}(g *0)$:
    \begin{enumerate}
        \item[(4)] $d_\Gamma(g\bullet e,\varphi^{-1}(g*0))<\lambda$.
    \end{enumerate}
    \end{enumerate}
\end{lemma}

The proof of each part of this lemma will use the same single strategy. We will begin by writing out the definition of $*$, and then we'll consecutively undo the outermost operation.

\begin{proof}
    Throughout this proof, let $g,h \in G$ and $x,y \in \R$ be arbitrary.
    
    Proof of $(1)$: Unpacking the definition of $*$, we have the following:
    \[d_\mathbb{R}((gh)*x,g*(h*x)) = d_\R(\varphi(gh\bullet \varphi^{-1}(x)), \varphi(g\bullet \varphi^{-1}(\varphi(h\bullet \varphi^{-1}(x)))))).\]
    Since $\varphi$ is a quasi-isometry, we have the following upper bound:
    \[d_\mathbb{R}((gh)*x,g*(h*x)) \leq \lambda\color{black} d_\Gamma(gh\bullet \varphi^{-1}(x), g\bullet \varphi^{-1}(\varphi(h\bullet \varphi^{-1}(x)))))+ \epsilon\]
    Since $G$ acts by isometries on $\Gamma$, we can apply $g^{-1}\bullet$ to the distance on the right-hand side and simplify to:
    \[d_\mathbb{R}((gh)*x,g*(h*x)) \leq \lambda d_\Gamma(h\bullet \varphi^{-1}(x), \varphi^{-1}(\varphi(h\bullet \varphi^{-1}(x))))+\epsilon\]

    Since $\varphi^{-1}$ and $\varphi$ are quasi-inverses, applying them one after another can only move a point at most $\epsilon$. Thus, with an application of the triangle inequality we can further simplify to:
    \begin{align*}
        d_\mathbb{R}((gh)*x,g*(h*x)) &\leq \lambda (d_\Gamma(h\bullet \varphi^{-1}(x), h\bullet \varphi^{-1}(x))+ \epsilon\color{black})+\epsilon\\
        &= \lambda (0+\epsilon)+\epsilon\\
        &= \lambda\epsilon+\epsilon,
    \end{align*}
    as desired.

    A directly analogous argument proves $(2)$ and $(4)$ as well.

    Proof of $(3)$: This proof will rely on the fact that $\varphi$ and $\varphi^{-1}$ are $(\lambda,\epsilon)$-quasi-isometries. From the definition, we know that:
    \[d_\R(g*x,g*y)=d_\R (\varphi^{-1}(g\bullet \varphi(x)), \varphi^{-1}(g\bullet \varphi(y))).\]

    Since $\varphi^{-1}$ is a quasi-isometry, we have the following bounds:
    \[\frac{1}{\lambda} d_\Gamma (g\bullet \varphi(x),g\bullet \varphi(y))-\epsilon\leq d_\R(g*x,g*y)\leq \lambda d_\Gamma (g\bullet \varphi(x),g\bullet \varphi(y))+\epsilon \]

    Since $G$ acts by isometries on $\Gamma$, we again apply $g^{-1}\bullet$ to both bounds get:
    \[\frac{1}{\lambda} d_\Gamma ( \varphi(x), \varphi(y))-\epsilon\leq d_\R(g*x,g*y)\leq \lambda d_\Gamma (\varphi(x), \varphi(y))+\epsilon \]

    Finally, since $\varphi$ is a quasi-isometry, we get that:
    \[\frac{1}{\lambda}\left( \frac{1}{\lambda}d_\R ( x, y)-\epsilon\right) -\epsilon\leq d_\R(g*x,g*y)\leq \lambda\left( \lambda d_\R ( x, y)+\epsilon\right)+\epsilon\]

    Which we simplify to:
    \[\frac{1}{\lambda^2}d_\R ( x, y)-\frac{\epsilon}{\lambda}-\epsilon\leq d_\R(g*x,g*y)\leq  \lambda^2 d_\R ( x, y)+\lambda\epsilon+\epsilon\]

    Completing our proof.
\end{proof}

As a brief aside, maps satisfying the properties $(A),(B),(C)$ are sometimes called \textbf{quasi-actions}. The interested reader can explore this concept further in \cite{Quasi_Actions_MLS}. In fact, the main result of that paper is another proof of the quasi-isometric rigidity of $\Z$, though in greater generality. 

This formulation of the map $*$ in the notation of a quasi-action and the proof idea for Lemma \ref{lem:NoFlip} below are both due to L\'eo Delage. The authors would like to thank him for his contributions to this project.
 
\section*{Step 2: An Action on the Ends} \label{action ends}
We have shown that this map $*$ assigns to each element $g \in G$ a quasi-isometry of $\R$. Moreover, isometries of $\R$ come in only two flavors: reflections and translations. So perhaps the map $*$ associates to each $g$ either a `reflection-like' or `translation-like' map on $\mathbb{R}$. This turns out to be true in the following sense:

\begin{lemma}\label{lem:ActionOnFarIntervals}
    Fix some $g \in G$. Then, for $x > \lambda^2(\lambda^2+\lambda\epsilon+3\epsilon+d_\R(g*0,0)+1)$, either:
    \begin{enumerate}
        \item (`Translation-like') $g*[x,\infty)\subseteq (0,\infty)$ or
        \item (`Reflection-like') $g*[x,\infty) \subseteq (-\infty,0)$
    \end{enumerate}
\end{lemma}

\begin{proof}
    We will prove this statement by showing that for each $n \in \mathbb{N}$, $g*[x, x+n] \subseteq (0, \infty)$ under the assumption that $g *x > 0$. The logic will be such that a totally analogous argument proves that if $g*x <0$, then $g*[x, x+n] \subseteq (-\infty, 0)$ for each $n \in \mathbb{N}$. For this to yield the desired statement, we need to rule out the third possibility, namely that $g *x = 0$. So fix some arbitrary  $x>\lambda^2(\lambda^2+\lambda\epsilon+\epsilon+d_\R(g*0,0)+1)$. 
    
    We will show that $0<d_\R(g*x,0)$. In fact, we will show the stronger statement that any $y \geq x$, $g*y$ must be far from zero.

    From the triangle inequality, we know that:
    \[d_\R(g*y,g*0)-d_\R(g*0,0)\leq d_\R(g*y,0)\]

    From Lemma \ref{lem:four properties}.3, the left-hand side is bounded below:
    \[\frac{1}{\lambda^2}d_\R(y,0)-\frac{\epsilon}{\lambda}-\epsilon-d_\R(g*0,0)\leq d_\R(g*y,0)\]

    In particular, using the facts that $y\geq x>\lambda^2(\lambda^2+\lambda\epsilon+3\epsilon+d_\R(g*0,0)+1)$ and $\frac{\epsilon}{\lambda}<\epsilon$, we can further bound this new left-hand side from below as well, obtaining:
    \begin{align*}
        \frac{1}{\lambda^2}(\lambda^2(\lambda^2+\lambda\epsilon+3\epsilon+d_\R(g*0,0)+1))-\frac{\epsilon}{\lambda}-\epsilon-d_\R(g*0,0)&\leq d_\R(g*y,0)
    \end{align*}

    Moreover, 

    \[  \lambda^2+\lambda\epsilon+3\epsilon+d_\R(g*0,0)+1-\epsilon-\epsilon-d_\R(g*0,0) =  \lambda^2+\lambda\epsilon+\epsilon+1. \]

    Thus, $\lambda^2+\lambda\epsilon+\epsilon+1 \leq d_\R(g*y,0)$ for all $y\geq x$. We can now conclude that $g*x\neq 0$.
    
    Now, assume that $g*x >0$, and assume further that $g*[x,x+n]\subseteq (0,\infty)$ for some $n \in \mathbb{N}$ (including the possibility that $n=0$). Then, we will show that $g*[x+n,x+n+1] \subseteq (0,\infty)$ as in Figure \ref{fig:Induction}.

    \begin{figure}[h]
    \centering
    \includegraphics[width=0.5\linewidth]{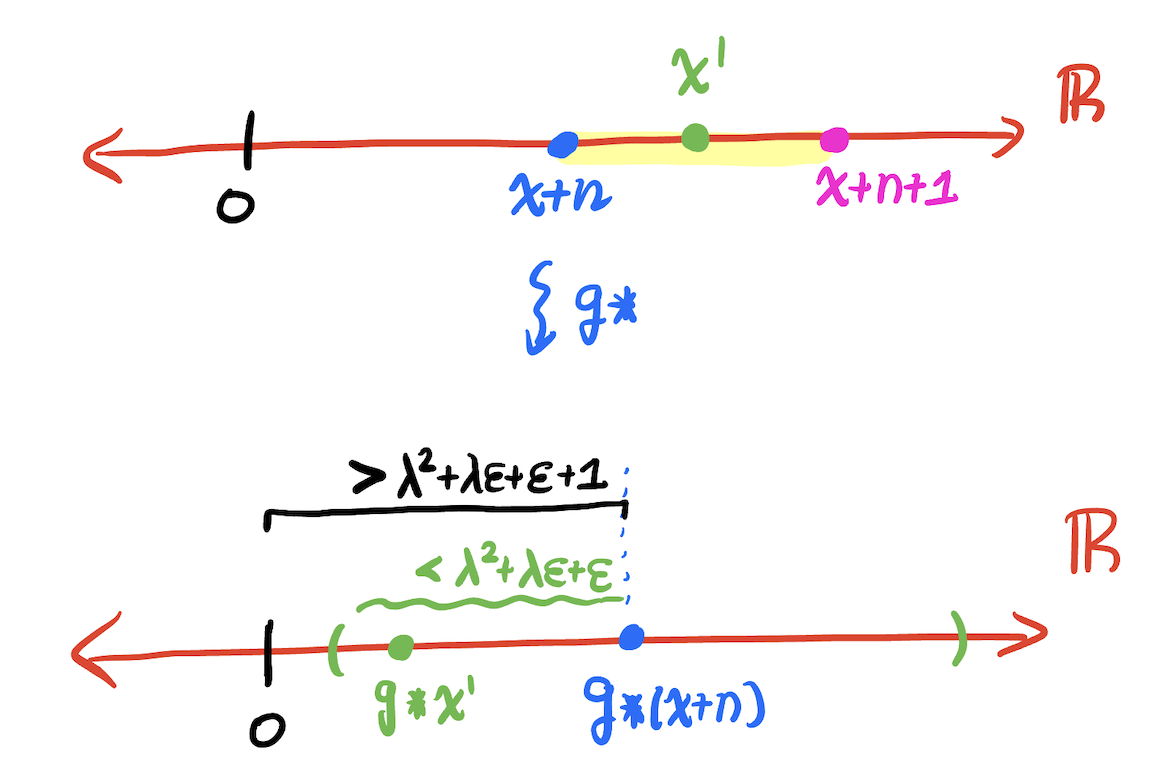}
    \caption{The inductive step of Lemma \ref{lem:ActionOnFarIntervals}}
    \label{fig:Induction}
    \end{figure}

    Since $x+n\geq x$, we know by our general statement above that $\lambda^2+\lambda\epsilon+\epsilon+1<d_\R(g*(x+n),0)$. And since we are assuming that $g*(x+n)>0$, it follows that
    \[0<\lambda^2+\lambda\epsilon+\epsilon+1<g*(x+n).\]

    Let $x' \in [x+n,x+n+1]$ be arbitrary. Then, by Lemma \ref{lem:four properties} we have that:
    \[d_\R(g*(x+n),g*x')\leq \lambda^2 d_\R(x+n,x')+\lambda\epsilon+\epsilon.\]

    Since $x' \in [x+n,x+n+1]$, we have that $d_\R(x+n,x')\leq 1$. Thus: 
    \[d_\R(g*(x+n),g*x')\leq \lambda^2 +\lambda\epsilon+\epsilon.\]

    Therefore, every element in $[x+n,x+n+1]$ must map to within $\lambda^2 +\lambda\epsilon+\epsilon$ of $g*(x+n)$. As we also know that $\lambda^2+\lambda\epsilon+\epsilon+1<g*(x+n)$. This allows us to conclude that every element in $g*[x+n,x+n+1]$ is positive, completing our inductive step.

    As we mentioned up above, a directly analogous argument holds for the case where we assume $g*[x,x+n]\subseteq (-\infty,0)$.

    Therefore, if we choose an $x$ sufficiently large, we have that either $g*[x,\infty)\subseteq (0,\infty)$ or $g*[x,\infty) \subseteq (\infty,0)$.
\end{proof}

Lemmas \ref{lem:four properties} and \ref{lem:ActionOnFarIntervals} can be combined to deduce that, as long as we ignore whatever happens inside of a fixed interval around $0$, the map $x \mapsto g *x$ either preserves each infinite ray $(-\infty, 0)$ and $(0, \infty)$, or permutes them: 

\begin{corollary}
    Given $x$ sufficiently large as in the statement of the previous lemma,
    \begin{enumerate}
        \item if $g*[x,\infty)\subseteq (0,\infty)$, then $g*(-\infty,-x]\subseteq (-\infty,0)$. 
        \item if $g*[x,\infty)\subseteq (-\infty,0)$, then $g*(-\infty,-x]\subseteq (0,\infty)$. 
    \end{enumerate}
\end{corollary}

Thus, we have determined that each element $g \in G$ either acts as a coarse translation, keeping each end\footnote{We use the word `ends' in an informal sense here, though a reader familiar with the definition of the ends of a topological space may note that this statement holds with the formal definition as well. See Appendix A for more on ends.} of the real line fixed; or, a coarse reflection, which flips the ends of the real line.  Let us use this fact to construct a new group action that only looks at this large scale behavior.

For notational ease, for some $g \in G$ let $l_g$ denote the bound $\lambda^2(\lambda^2+\lambda\epsilon+3\epsilon+d_\R(g*0,0)+1)$ given in the lemma above.

Define the map $*:G\times \{\pm \infty\} \to \{\pm \infty\}$ by the following. For each $g \in G$, choose any $x>l_g$. If $g*[x,\infty)\subseteq (0,\infty)$ then $g*\pm \infty=\pm\infty$. Otherwise, $g*\pm \infty=\mp\infty$.

\begin{lemma}
    The map $*:G\times \{\pm \infty\} \to \{\pm \infty\}$ is a well-defined group action.
\end{lemma}

\begin{proof}
    Lemma \ref{lem:ActionOnFarIntervals} above demonstrates that this map $*$ is well defined.

    For the identity axiom, consider the map $e*$ on the real line. From Lemma \ref{lem:four properties}.2, we know that $d_\mathbb{R}(e*x,x)\leq \epsilon$ for all $x \in \mathbb{R}$. Since $e*$ can only move elements in $\mathbb{R}$ by up to $\epsilon$, it must map every real number $x>\epsilon$ to a positive number. Therefore, $e*$ does not flip the ends of the real line, and $e*\pm \infty=\pm\infty$, as desired.

    It remains to show that $(gh) *= g*(h*)$.
    
    Consider the case where $h*\infty =\infty$ and $g*\infty=\infty$ (which means that $h*-\infty=-\infty$ and $g*-\infty=-\infty$). All other cases will follow by an analogous argument.
    
    Choose $x$ so that $x>\max(l_h,l_{gh})$ and $d_\R(h*x,0)> l_g$.

    Since $h*x>0$, $d_\R(h*x,0)>l_g$, and $g*\infty=\infty$, it follows by definition of $l_g$ that $g*(h*x)>0$. In fact, from the proof of Lemma \ref{lem:ActionOnFarIntervals} above, we get the stronger statement that $g*(h*x)>\lambda^2+\lambda\epsilon+\epsilon+1$.
    
    From Lemma \ref{lem:four properties}, we know that $d_\R((gh)*x,g*(h*x))\leq \epsilon+\lambda\epsilon$. Therefore, if $g*(h*x)$ is positive, $(gh)*x$ must be positive as well.

    Since $x>l_{gh}$, the fact that $(gh)*x$ is positive implies by definition that $gh*\infty=\infty$ and $gh*-\infty=-\infty$, proving the desired equality of maps.

\end{proof}

This group action gives rise to a homomorphism $f:G \to \text{Isom}_{Set}(\{\pm \infty\})\cong \mathbb{Z}_2$. The kernel $\ker(f)$ is an infinite, index 2 subgroup of $G$ of coarse translations on $\mathbb{R}$. This subgroup will prove to be very useful for us. Recall that we hope to use the geometry of $\R$ to find a copy of $\Z\leq G$. As a non-trivial translation of $\R$ is an infinite order element, it may be easier to find the desired copy of $\Z$ within the subgroup $\ker(f)$.

In fact, we will now show that elements of $\ker(f)$ act like translations in the following sense: any $g \in \ker(f)$ preserves the order of sufficiently far apart elements in $\R$. We begin this proof with a helper lemma.

\begin{lemma}\label{lem:PassClose}
    Given an arbitrary $g \in G$, $g*([a,b])$ passes within distance $l=\lambda^{2}+\lambda\epsilon+\epsilon$ of every point on the interval $[g * a,g* b]$ (or $[g* b,g * a]$ if the order is flipped).
\end{lemma}

\begin{proof}
    
    Without loss of generality assume $g * a<g * b$. Let $x \in [g * a,g* b]$ be arbitrary.

    In the event that $g * a$ or $g * b$ is within distance $\frac{\lambda^{2}+\lambda\epsilon+\epsilon}{2}$ of $x$, the conclusion immediately holds because $g *a$ is in both $[g*a, g*b]$ and $g*[a,b]$ (and similarly for $g*b$.

    So let us assume that this doesn't happen, namely that $x$ is greater than distance $\frac{\lambda^{2}+\lambda\epsilon+\epsilon}{2}$ from both $g*a$ and $g*b$.

    Let us define two subsets $M,N$ of $[a,b]$ as follows:
    \[ M = \left\{m \in [a,b]: g*m < x-\frac{\lambda^{2}+\lambda\epsilon+\epsilon}{2} \right\},  \]
    and 
    \[ N = \left\{ n \in [a,b]: x+\frac{\lambda^{2}+\lambda\epsilon+\epsilon}{2} < g * n \right\}.\]
    Note that we have chosen $x$ so that neither $M$ nor $N$ can be empty. 

    For an arbitrary element $m \in M$ and $n \in N$, we use the fact that $g*m<x<g*n$ to deduce that
    \[d_\mathbb{R}(g*m,x)+d_\mathbb{R}(x,g*n)=d_\mathbb{R}(g*m,g*n)\]

    Since we know that both $g*m$ and $g*n$ are greater than distance $\frac{\lambda^{2}+\lambda\epsilon+\epsilon}{2}$ from $x$, we have that:
    \[\lambda^{2}+\lambda\epsilon+\epsilon<d_\mathbb{R}(g*m,g*n)\]

    By Lemma \ref{lem:four properties}, we get the upper bound
    \[d_\mathbb{R}(g*m, g*n) \leq \lambda^2d_\mathbb{R}(m,n)+\lambda\epsilon+\epsilon,\]
    and therefore 
    \[ \lambda^{2} + \lambda \epsilon + \epsilon < \lambda^{2} d_{\R}(m,n) + \lambda \epsilon + \epsilon.\]
    A quick algebraic manipulation reveals that
    \[1<d_\mathbb{R}(m,n)\]

    Since there is no way to partition a closed interval into two nonempty sets such that every element of one is distance at least $1$ from every element of the other, there must be a point in neither $M$ nor $N$. In other words, $g*$ must map some point in $[a,b]$ close to (within distance $l$ of) our arbitrary point $x$.
\end{proof}

Now, we can prove the order preserving quality of $g*$. 

\begin{lemma}\label{lem:NoFlip}
    For all real numbers $n>\lambda^2 l+\lambda\epsilon+\lambda^2\epsilon$ and group elements $g \in \ker(f)$, the following inequality holds: $g *x < g*(x+n)$.
\end{lemma}

\begin{proof}
    Assume, for contradiction, that $g*(x+n)\leq g*x$. Since $g \in \ker(f)$, there must be some $y>x+n$ large enough so that $g*(x+n)\leq g*x<g*y$.

    We can then use Lemma \ref{lem:PassClose}, which states that since $g*x$ lies in the interval $[g*(x+n),g*y]$, there must be some $x+n' \in [x+n,y]$ so that $d_\mathbb{R}(g*(x+n'),g*x)<l$.

    By Lemma \ref{lem:four properties}, we get further that:
    \[\frac{1}{\lambda^2}d_{\mathbb{R}}(x+n',x)-\frac{1}{\lambda}\epsilon-\epsilon\leq d_\mathbb{R}(g*(x+n'),g*x)<l\]

    Which we algebraically manipulate to:
    \[n'<\lambda^2l+\lambda\epsilon+\lambda^2\epsilon\]

    But since $x+n' \in [x+n,y]$, we know that $n'>n$, leading to the statement:
    \[n<\lambda^2l+\lambda\epsilon+\lambda^2\epsilon \]
    contradicting our given lower bound on $n$.

    Therefore, $g*x<g*(x+n)$ for $n>\lambda^2 l+\lambda\epsilon+\lambda^2\epsilon$.
\end{proof}

\begin{figure}[h]
    \centering
    \includegraphics[width=0.5\linewidth]{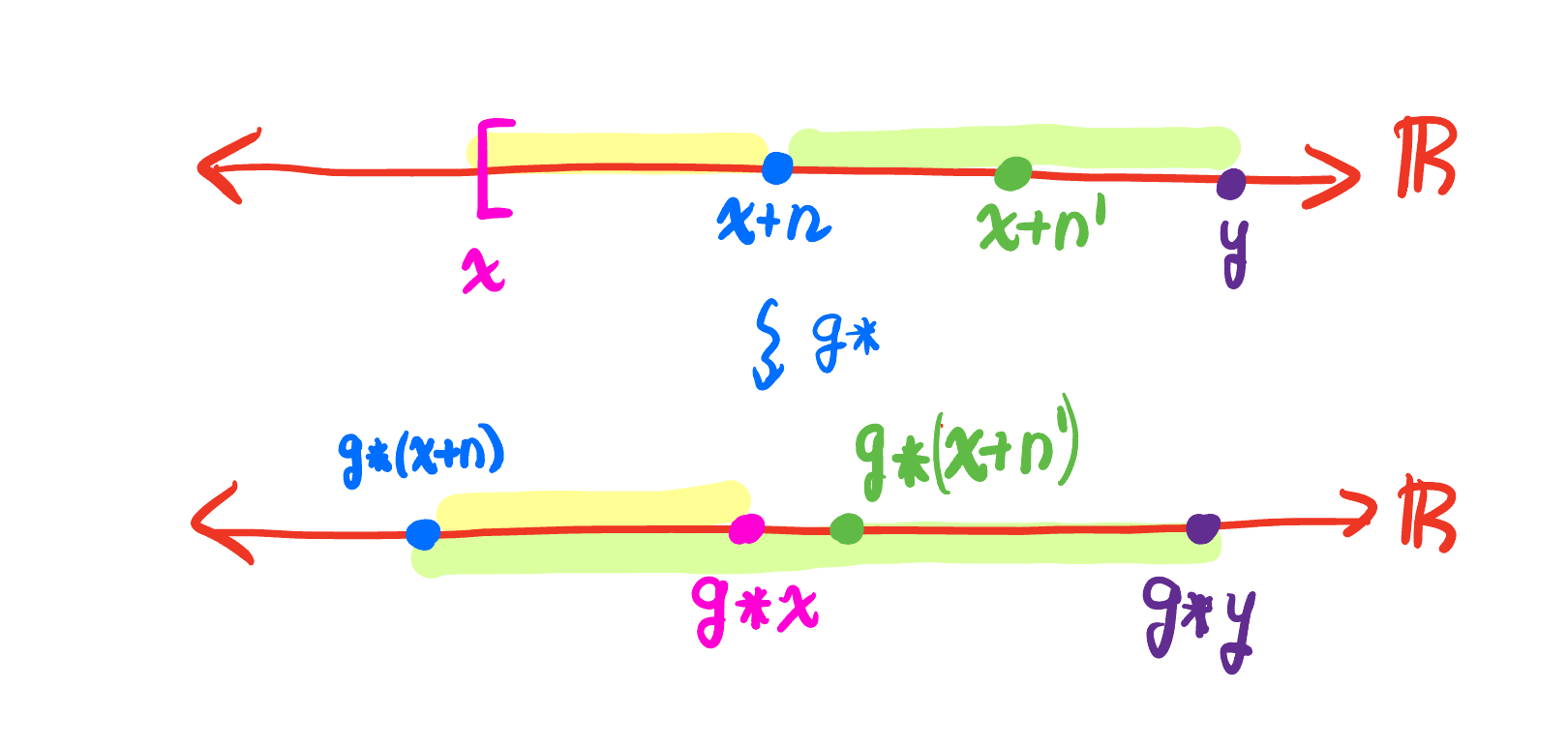}
    \caption{The contradiction in Lemma \ref{lem:NoFlip}. The points $x$ and $x+n'$, which have large distance in the domain, are mapped close to each other in the codomain, contradicting the fact that $g*$ is a quasi-isometry.}
    \label{fig:FoldingLemma}
\end{figure}

\section*{Step 3: Finding a Copy of $\Z$} \label{copy}
At the end of the previous section, we demonstrated that $g*$ preserves the order of elements of $\mathbb{R}$ that are at least distance $n$ apart, with this constant $n$ not depending at all on our choice of $g \in \ker(f)$. Therefore, provided our first application of $g*$ moves zero sufficiently far away, continuing to apply $g*$ will only take us further and further from zero, thus demonstrating that $g$ is an infinite order element. We formalize this in the following lemma.

\begin{lemma}\label{lem:Finding_g}
    There exists an element $g \in G$ so that for all $z \in \mathbb{Z}$:
    \begin{enumerate}[(i)]
        \item The distance $d_\mathbb{R}(g^z*0,g^{z+1}*0)$ is bounded away from $0$ and is uniformly bounded above.
        \item $g^z*0<g^{z+1}*0$
        
    \end{enumerate}
\end{lemma}

$(i)$ ensures that subsequent elements of $\{g^z*0\}_{z\in \Z}$ are fairly uniformly spaced and never accumulate. $(ii)$ shows that applying $g*$ genuinely translates $0 \in \R$ in the positive direction, which will be used later to prove that $g$ is infinite order. These two properties are precisely what we need for $\{g^z*0\}_{z\in \Z}$ to look like the integers sitting in $\R$.

\begin{figure}[h]
    \centering
    \includegraphics[width=0.5\linewidth]{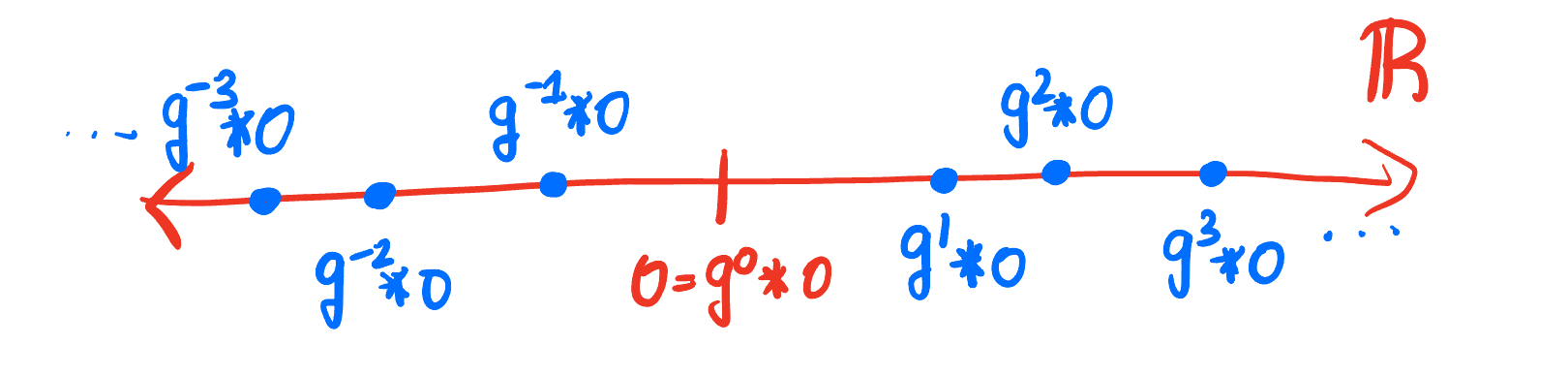}
    \caption{The orbit $\{g^z*0\}_{z \in \Z}$ looks almost like an isometric copy of $\Z$ in $\R$, but with some distortion.}
    \label{fig:QuasiDenseOrbitInR}
\end{figure}

\begin{proof}
    Since $\ker(f)$ is an infinite subgroup of a finitely generated infinite group, we can always choose a $g \in \ker(f)$ which makes $d_\Gamma(e,g\bullet e)$ as large as we like. 

    In particular, let us choose $g$ so that $0<g*0$ and so that:
    \[d_\Gamma(e,g\bullet e)>  \lambda\left(\epsilon+\lambda^2\left(\frac{\epsilon}{\lambda}+\lambda\epsilon+2\epsilon+n\right)\right)\] with $n$ as in Lemma \ref{lem:NoFlip}.

    Now, let us prove that (i) holds for this $g$.

    First, let us find a lower bound. For $z \in \Z$ arbitrary, we have the following from the triangle inequality, repeated uses of Lemma \ref{lem:four properties}, and our bound on $d_\Gamma(e,g\bullet e)$ above (we use red to highlight what changes from line to line):

    \begin{align}
        d_\mathbb{R}(g^z*0,g^{z+1}*0)&\geq d_\mathbb{R}(g^z*0,g^{z}*(g*0)) -d_\mathbb{R}(g^z*(g*0),g^{z+1}*0)\\
        &\geq d_\mathbb{R}(g^z*0,g^{z}*(g*0))\color{red}-\epsilon-\lambda\epsilon\\
        &\geq \color{red}\frac{1}{\lambda^2}d_\mathbb{R}(0, g*0)-\frac{\epsilon}{\lambda}-\epsilon\color{black}-\epsilon-\lambda\epsilon\\
        &= \frac{1}{\lambda^2}d_\mathbb{R}(\color{red}\varphi(e), \varphi(g\bullet \varphi^{-1}(0))\color{black})-\frac{\epsilon}{\lambda}-\epsilon-\epsilon-\lambda\epsilon\\
        &= \frac{1}{\lambda^2}d_\mathbb{R}(\varphi(e), \varphi(g\bullet \color{red}e\color{black}))-\frac{\epsilon}{\lambda}-\epsilon-\epsilon-\lambda\epsilon\\
        &\geq \frac{1}{\lambda^2} \left(\color{red}\frac{1}{\lambda}d_{\Gamma}(e, g\bullet e)-\epsilon\color{black}\right)-\frac{\epsilon}{\lambda}-\epsilon-\epsilon-\lambda\epsilon\\
        &> n
    \end{align}

    An analogous argument gives a uniform upper bound on $d_\mathbb{R}(g^z*0,g^{z+1}*0)$ as well, demonstrating that $(i)$ holds.

    Now, let us demonstrate that this same $g$ satisfies condition $(ii)$. Again, let $z \in \Z$ be arbitrary.
    
    Revisiting our computation above, lines $(3)$ and $(7)$ state in particular that that $d_\mathbb{R}(0,g*0)> n$. 
    
    Thus, $0$ and $g*0$ are sufficiently far apart in $\mathbb{R}$ for Lemma \ref{lem:NoFlip} to apply, which gives us that $g^z*0< g^z*(g*0)$. 

    We must now ensure that passing from $g^z*(g*0)$ to $g^{z+1}*0$ does not flip the inequality. We have shown that the distance between $g^z*0$ and $g^{z+1}*0$ is always greater than $n$, and from Lemma \ref{lem:four properties}.1, we know that the distance between $g^z*(g*0)$ and $g^{z+1}*0$ is at most $\epsilon+\lambda\epsilon$. By the definition of $n$, we know that $n>\epsilon+\lambda\epsilon$. Therefore, the inequality does not flip, and we can conclude that $g^z*0<g^{z+1}*0$. Property $(ii)$ holds.
\end{proof}

As a corollary, we have that $g$ is an infinite order element, as there is no way to compose nontrivial positive translations of $\R$ and get back to the identity. We prove this below:

\begin{corollary}\label{cor:InfiniteOrder}
    The element $g$ is infinite order.
\end{corollary}

\begin{proof}
    Since $g^z*0<g^{z+1}*0$, and each step $d_\mathbb{R}(g^z*0,g^{z+1}*0)$ is bounded below by $n$, we know that $|z|n\leq d_\R(0,g^z*0)$.

    If $g^z=e$ for any $z>0$, then this would imply by Lemma \ref{lem:four properties}.2 that $d_\R(0,g^z*0)<\epsilon$. Thus, we would have that for some $z>0$ a positive integer, $zn<\epsilon$. This cannot occur since $n>\epsilon$.

    Thus, $g$ must be infinite order.
\end{proof}

This solves our concern at the beginning of this argument. By using the quasi-isometry to $\R$, it is possible to show that this group $G$ (which at first we knew none of the group theoretic structure of) must contain a copy of $\Z$, namely $\langle g \rangle$. It remains to show that this copy of $\Z$ is finite index in $G$, which we will do by finding that $\{g^z\bullet e\}_{z \in \Z}$ is quasi-dense in the Cayley graph $\Gamma$.

First, we have a as a corollary to Lemma \ref{lem:Finding_g} that:

\begin{corollary}
    Given $g$ as given in Lemma \ref{lem:Finding_g}. The set $\{g^z*0\}_{z \in \Z}$ is quasi-dense in $\mathbb{R}$.
\end{corollary}

\begin{proof}
    From Lemma \ref{lem:Finding_g} above, we know that $\{g^z*0\}_{z \in \Z}$ is a strictly ascending bi-infinte sequence in $\mathbb{R}$. Furthermore, since subsequent elements are least distance $n$ apart with $n>0$, we know that $d_\mathbb{R}(0,g^z*0)$ is eventually greater than any real number for some large enough $z$.

    Therefore, for an arbitrary $x \in \mathbb{R}$, we have that $g^{N}*0\leq x <g^{N+1}*0$ for some $N \in \Z$. Again by Lemma \ref{lem:Finding_g}, we know that $d(g^{N}*0,g^{N+1}*0)$ is uniformly bounded above by a constant $n'$. Therefore, $d_\mathbb{R}(g^N*0,x)<n'$.

    Thus, every $x \in \mathbb{R}$ is within distance $n'$ of some element in $\{g^z*0\}_{z \in \Z}$, which proves the quasi-density of our subset.
\end{proof}

As quasi-density of a subset is a property that is preserved by quasi-isometry, we can apply $\varphi^{-1}$ and further have that:
\begin{corollary}\label{cor:QuasiDense}
    The set $\{\varphi^{-1}(g^z*0)\}_{z \in \Z}$ is quasi-dense in $\Gamma$.
\end{corollary}

We have shown that every element of $\Gamma$ is within bounded distance of some element of $\varphi^{-1}(g^z *0)$. From Lemma \ref{lem:four properties}.4, we know that if we take an element in $\{\varphi^{-1}(g^z*0)\}_{z \in \Z}$ and its corresponding element in $\{g^z\bullet 0\}_{z \in \Z}$, they differ by at most distance $\lambda$. Thus, if $\{\varphi^{-1}(g^z*0)\}_{z \in \Z}$ is quasi-dense in $\Gamma$, then $\{g^z\bullet 0\}_{z \in \Z}$ must be as well.

Our original theorem statement now follows:

\begin{figure}[b]
    \centering
    \includegraphics[width=0.5\linewidth]{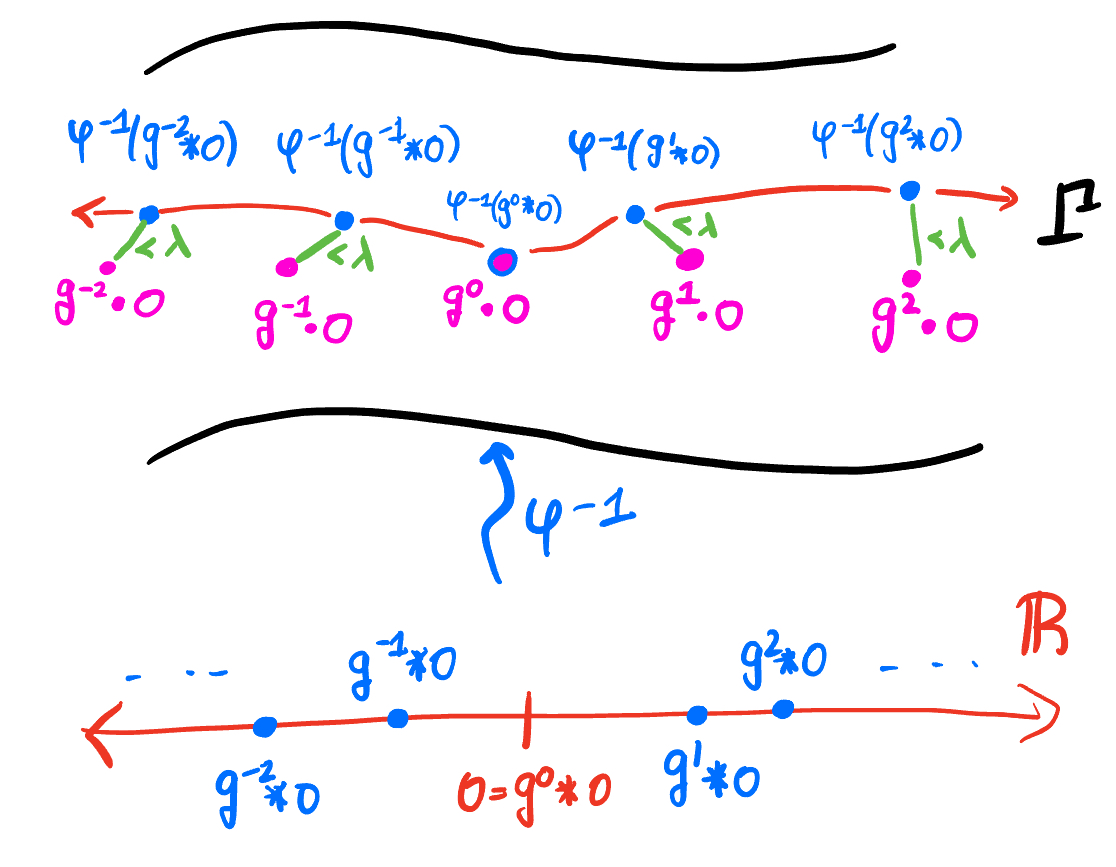}
    \caption{The proof of the main theorem. We find a copy of $\Z$ in $G$ by finding a quasi-dense `copy' of $\Z$ in $\R$ and then sending it back into $\Gamma$ using the quasi-isometry.}
    \label{fig:VirtuallyZ}
\end{figure}

\begin{proof}[Proof of Theorem \ref{thm:main}]   
    Choosing a $g \in \ker(f)$ sufficiently large as in Lemma \ref{lem:Finding_g}, we have from Corollary \ref{cor:QuasiDense} and Lemma \ref{lem:four properties}.4 that the set $\{g^z\bullet 0\}_{z \in \Z}$ is quasi-dense in $\Gamma$. 
    
    Now, let us show that due to quasi-density, we have a \textit{finite-index} subgroup of $G$. Let $G/\langle g \rangle$ denote the set of right cosets of $\langle g \rangle$. Quasi-density of $\langle g \rangle$ says that there is some $k \ge 0$ so that every element of $\Gamma$ lies within distance $k$ of some element of $\langle g \rangle$.  

    The ball $B(k,e)$ of radius $k$ about $e \in \Gamma$ contains at most $|S|^{k}$ vertices, with $|S|$ the cardinality of the generating set. We need to show that every coset of $G/\langle g \rangle $ has a representative in this finite set of vertices, with no two distinct cosets sharing a vertex.
    
    Fix some arbitrary coset $\langle g \rangle g' \in G/\langle g \rangle$. All cosets are images of one another under left multiplication by various group elements, which are isometries of $\Gamma$. Thus, every coset is $k$-quasi dense, and there must be an element in $B(k,e) \cap \langle g \rangle g'$. Since cosets are disjoint, it follows that this element in $B(k,e)$ must be distinct for different cosets. Therefore, there are at most $|S|^{K}$ cosets, as desired. 
    
    This implies that the subgroup $\langle g \rangle \cong \Z$ (by Corollary \ref{cor:InfiniteOrder}) is finite index in $G$, and thus that $G$ is virtually $\Z$.
\end{proof}

\section*{Appendix A: Alternative Proof Methods}\label{sec:AppA}
This classical result is sometimes described as the two-ended case of Stallings' Ends Theorem, which gives a general correspondence between the number of ends of a finitely generated group and its group structure. In the above proof, we seek to construct a concrete and accessible argument that relies only upon the definitions of a group action and a quasi-isometry.

Intuitively speaking, the \emph{ends} of a finitely generated groups are distinct ways you can go towards infinity in the group. A bit more concretely, the number of ends of a finitely generated group $G$ is the least upper bound (possibly infinity) of the number of unbounded connected components of the space obtained by removing a sequence of concentric closed, finite balls around the origin in a Cayley graph of $G$. Thus all finite groups have 0 ends, $\mathbb Z$ has two ends, and $\mathbb Z\oplus \mathbb Z$ has one end. A finitely generated non-abelian free group has infinitely many ends because as we remove larger and larger finite balls around the identity, we get more and more unbounded branches in the complement. 

It is a famous theorem of Hopf \cite{MR10267}, that any finitely generated group $G$ has 0, 1, 2, or infinitely many ends. Stallings later proved that a group with infinitely many ends can be split into pieces in a particular way. \cite{MR228573}. We use one of Flame's original pictures in Figure \ref{fig:Stallings} to indicate the general proof of the Hopf theorem in the case where more then two implies infinitely many ends. 

\begin{figure}[h]
    \centering
    \includegraphics[width=0.5\linewidth]{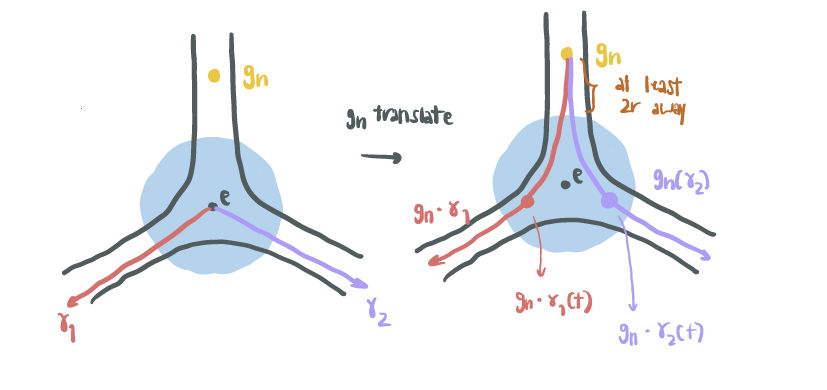}
    \caption{Flame's Sketch of the Greater than 2 Ends Portion of Hopf's Theorem in \cite{bridson_metric_1999}}
    \label{fig:Stallings}
\end{figure}

For readers with experience in metric geometry, an alternative proof method is outlined in Section I.8 of \cite{bridson_metric_1999}. This proof relies on the following facts:
\begin{enumerate}
    \item The number of ends of a space is a quasi-isometry invariant, so the graph $\Gamma$ must be two-ended.
    \item A group of isometries of a space induces a group action on the ends of a space.
    \item In a proper geodesic metric space, the ends of a space can be defined as an equivalence class of proper geodesic rays issuing from \textit{any} point in the space.
\end{enumerate}

\begin{proof}[Proof Sketch]
    For this proof, one can remove a ball of radius $k$ in $\Gamma$ about the identity that separates the graph into exactly two unbounded components that determine the two ends of $\Gamma$, and fix the finite index subgroup $H \leq G$ that acts as the identity on the set of two ends. Choose a point $g\bullet e$ with $g \in H$ which lies one of these two components, which we will call $E$. Without loss of generality, choose $g$ so that it has word length $k$. We can always increase the radius of the separating ball so that such an element can be chosen.

    Then, let us assume for contradiction that $g$ is finite order $n>1$. Since each application of $g\bullet$ moves a point exactly distance $k$, the orbit $\{g^k\bullet e\}_{1\leq k\leq n-1}$ must lie entirely in $E$.
    
    As in Figure \ref{fig:BHProof}, choose a geodesic ray $r$ based at $g^{n-1} \bullet e$ going towards the end that this point is contained in. Note that by our construction of $g$, the basepoint $g^{n-1}$ must lie at distance $k$ from the identity. Then, $g\bullet r$ will be a geodesic ray based at the identity heading towards the same end as $r$ (this is since $g \in H$). Thus, there must be a point $g\bullet r(x)$ with $x>k$ that passes within distance $k$ of $g^{n-1}\bullet e$. Applying the isometry $g\bullet$ to the pair of points $g\bullet r(x)$ and $g^{n-1}\bullet e$ will force them to be at least distance $k$ apart, giving us our contradiction.

Then, by using the quasi-isometry to $\R$, it is possible to show that $\{g^z\bullet e\}_{z \in \Z}$ is quasi-dense in $\Gamma$, completing this proof.
\end{proof}

\begin{figure}[h]
    \centering
    \includegraphics[width=0.45\linewidth]{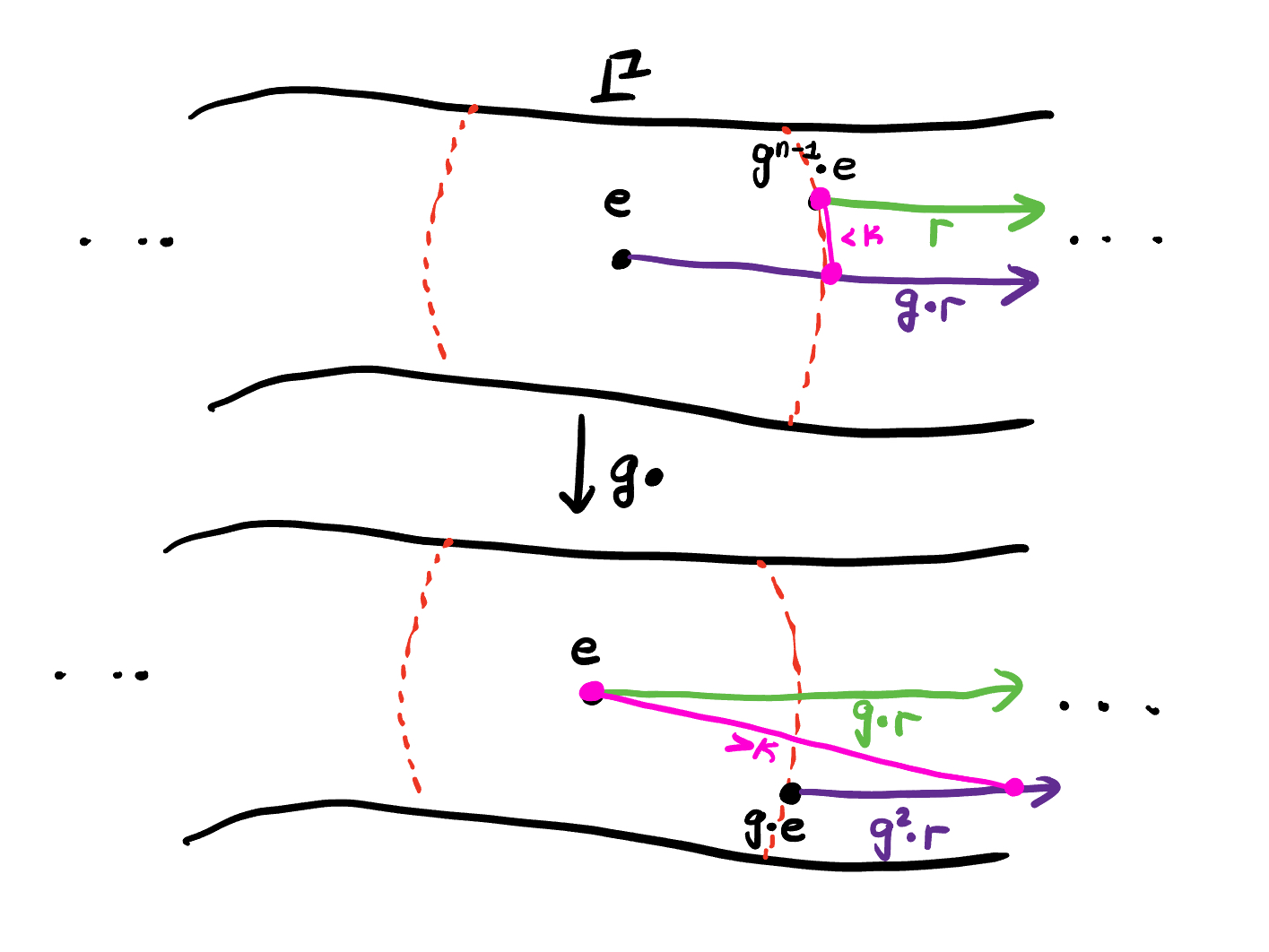}
    \caption{Points on the two rays are forced far apart by the isometry $g\bullet$}
    \label{fig:BHProof}
\end{figure}

If we change the assumption of `$\Gamma$ is quasi-isometric to $\R$' to `$\Gamma$ has two ends', then a proof can be found in Section 11.6 of \cite{meier_groups}. This proof uses the symmetric difference operation:
\[A \triangle B = (A \cup B)\setminus(A\cap B).\] 

We present a brief outline below, omitting technical details.

\begin{proof}[Proof Sketch]
    As above, choose a ball that separates $\Gamma$ into exactly  two unbounded components, and let $E$ be one of the two components. 

    Then, it is shown that the set $H=\{g \in G \;|\; E \triangle gE \text{ is finite}\}$ is a subgroup of index at most two. This corresponds to the subgroup that fixes the ends of the space.

    One can then construct the map:
    \begin{align*}
        \phi: H &\to \Z\\
        g &\mapsto |E \cap gE^c|-|E^c\cap gE|
    \end{align*}

    Where $|\cdot|$ is the number of vertices in the corresponding set. Using the fact that $E \triangle gE$ is finite for each $g \in H$, one can show that $|E \cap gE^c|$ and $|E^c\cap gE|$ are both finite, and thus this map is well defined.

    By partitioning and recombining the sets $E \cap gE^c$ and $E^c\cap gE$, one can show that this map $\phi$ gives a genuine group homomorphism onto $\Z$ with finite kernel. These arguments amount to showing that two sets have the same cardinality, which can be shown using purely set theoretic methods.

    Then, an application of the first isomorphism theorem and Lagrange's theorem proves that $G$ is virtually $\Z$.
\end{proof}

We conclude with a brief note on our proof. Many of our claims about the map $g*$ in Steps 1 and 2 can be shown to be special cases of more general facts about $\delta$-hyperbolicity (see Section III.1 \cite{bridson_metric_1999}), which is a property preserved under quasi-isometry. In particular, $\R$ is a $\delta$-hyperbolic space.

We have chosen to not invoke $\delta$-hyperbolicity, but the consequences of this central concept in geometric group theory are quite powerful. An interested reader can construct shorter alternate proofs of several of our lemmas using the facts that:
\begin{enumerate}
    \item The image of $\R$ under a quasi-isometry is called a quasi-geodesic.
    \item Quasi-isometries of a space send quasi-geodesics to quasi-geodesics.
    \item Quasi-geodesics in $\delta$-hyperbolic spaces stay `close' to geodesics (Theorem III.1.7 \cite{bridson_metric_1999}). 
\end{enumerate}

\section*{Appendix B: Extending to Groups of Linear Growth}
In a MathOverflow post \cite{MathOverflowPost}, Thurston describes a sketch for how one might prove that $\Gamma$ has a finite index subgroup isomorphic to $\mathbb{Z}$ under the weaker assumption that $\Gamma$ has \textit{linear growth}, which means that the number of elements in a ball of radius $k$ around the identity $e$ grows at most linearly in $k$:
\[ \limsup_{k \rightarrow \infty} \frac{|B(k,e)|}{k} < \infty.\]
This is indeed weaker than our assumption, since the existence of a quasi-isometry to $\mathbb{R}$ implies linear growth. Of course, since we claim it is possible to prove that linear growth implies that $\Gamma$ is virtually $\mathbb{Z}$, it follows a posteriori that in our context, linear growth implies a quasi-isometry to $\mathbb{R}$ but this is by no means obvious. 

We outline the contours of Thurston's sketch here: 

Let $G= (V,E)$ be a directed graph for which each edge $e$ is given a weight equal either to $0$ or $1$. Then a $\left\{0,1\right\}$-\textbf{flow} from vertex $v$ to vertex $u$ is a function $f: E \rightarrow \left\{0,1 \right\}$ so that 
\begin{enumerate}
\item For each edge $e$, $f(e) \leq w(e)$ (no edge can support a flow that is larger than its weight); 
\item For all $v' \notin \left\{v,u \right\}$, one has 
\[  \sum_{w: (w,v') \in E} f(w,v') - \sum_{w: (v', w) \in E} f(v',w) = 0.\]
(The flow into $v'$ exactly balances the flow out of $v'$, such that the net flow is $0$.)
\end{enumerate}

The \textbf{magnitude} of a flow, denoted $|f|$, is simply the sum of its values on all edges emanating from $v$:
\[ |f| = \sum_{(v,v') \in E} f(v,v'). \]
The properties of a flow imply that the magnitude can also be calculated by summing its values on all edges entering $u$:
\[ |f| = \sum_{v', u\in E} f(v', u).\]

A \textbf{cut} relative to $\left\{v,u \right\}$ of $G$ is a partitioning $(A,B)$ of $V$ so that $v \in A$ and $u \in B$. The \textbf{cut set} $C_{(A,B)}$ of a cut $(A,B)$ is the subset of edges $E$ connecting a vertex of $A$ to a vertex of $B$:
\[ C(A,B) = \left\{ (v_{1}, v_{2}) \in E: v_{1} \in A, v_{2} \in B \right\}. \]

Finally, the \textbf{magnitude of a cut} is equal to the sum of the weights of all edges in the associated cut-set:
\[ |(A,B)| = \sum_{e \in C_{(A,B)}}w(e). \]

The celebrated \textbf{max-flow min-cut theorem} states that the \textit{maximum} magnitude of a flow from $v$ to $u$ is equal to the \textit{minimum} magnitude of a cut relative to $\left\{v, u \right\}$. 

\begin{remark} \label{rem:intuition}
To get some intuition for why these two quantities should be related at all, one can imagine the flow as a ``pulse'' starting at $v$ and ending at $u$. At each stage of the pulse, the flow ``arrives'' at a new cut set, obtained from the previous one by just pushing the flow forward from each vertex $x$ to the vertices that are connected to it by edges emanating from $x$. One can for instance visualize this by imagining that a large collection of tiny stones begins at $v$. Then if edges $e_{1},..., e_{n}$ emanate from $v$ and arrive at vertices $v_{1},..., v_{n}$ respectively, we send a fraction of $w(e_{i})/\sum_{j}w(e_{j})$ of the total number of stones to vertex $v_{i}$. We continue this again and again. The magnitude of the flow is preserved at each stage of the pulse because of the requirement that the net flow remains $0$. This means that we can easily find many cut sets whose magnitude serves as an upper bound for the magnitude of the flow. Indeed, the set of vertices with at least one stone at a given stage comprises a cut set, and the total number of stones must therefore be at most the total weight summed over the vertices in that cut set.   
\end{remark}

With all of this in mind, assume that $\Gamma$ has linear growth and let $v,u$ be a pair of vertices. After assigning each edge of $\Gamma$ a weight of $1$, we can then ask for the maximum magnitude of a flow from $v$ to $u$. The max-flow min-cut theorem says that this is precisely equal to the minimum number of edges required to separate $v$ from $u$. 

Here is the key point: if $B$ is any closed ball about $v$ of some radius $k$ chosen to be less than the distance between $v$ and $u$, then $\partial B$ separates $v$ from $u$. Moreover, linear growth of $\Gamma$ implies that 
\[ \liminf_{k} |\partial B(k, v)| < \infty. \]
To see this, note that
\[ |B(k,v)| = \sum_{j=1}^{k} |\partial B(j,v)|,\]
and linear growth is the statement that the result of dividing the left hand side by $k$ is bounded in the limit. Therefore, the average size of the first $k$ spheres about $v$ is bounded in the limit, and so at least one of those spheres must have bounded size independent of $k$. It follows that the maximum magnitude of a flow is uniformly bounded from above, independent of the distance between $v$ and $u$. 

Now, suppose $u$ and $v$ are very far apart; then there are very many spheres about $u$ of varying radii that separate $u$ from $v$. Moreover, we can pick these spheres such that each has at most some $N$ number of elements, since as we've said above, spheres of bounded size (but with arbitrarily large radius) are abundant.  

We are thus in a situation where, by choosing $u$ and $v$ sufficiently far apart and a non-zero flow from $u$ to $v$, there is a number of spheres $S_{1},..., S_{n}$ about $u$ not containing $v$ and satisfying the following properties:
\begin{itemize}
\item Each $|S_{i}|$ is uniformly bounded above by $N$. 
\item For each $i$, the sum of flow values -- taken over all elements in $S_{i}$-- is uniformly bounded above by some $N'$.
\end{itemize}
Since $N$ and $N'$ are uniform, by increasing $d(u,v)$ we can choose $n$ to be as large as we want without changing $N,N'$. We now use the fact that given $N,N'$, there are only finitely many possible flows on a graph with at most $N$ vertices and with magnitude at most $N'$ (up to flow-preserving graph isomorphism). 

\begin{figure}[h]
    \centering
    \includegraphics[width=0.4\linewidth]{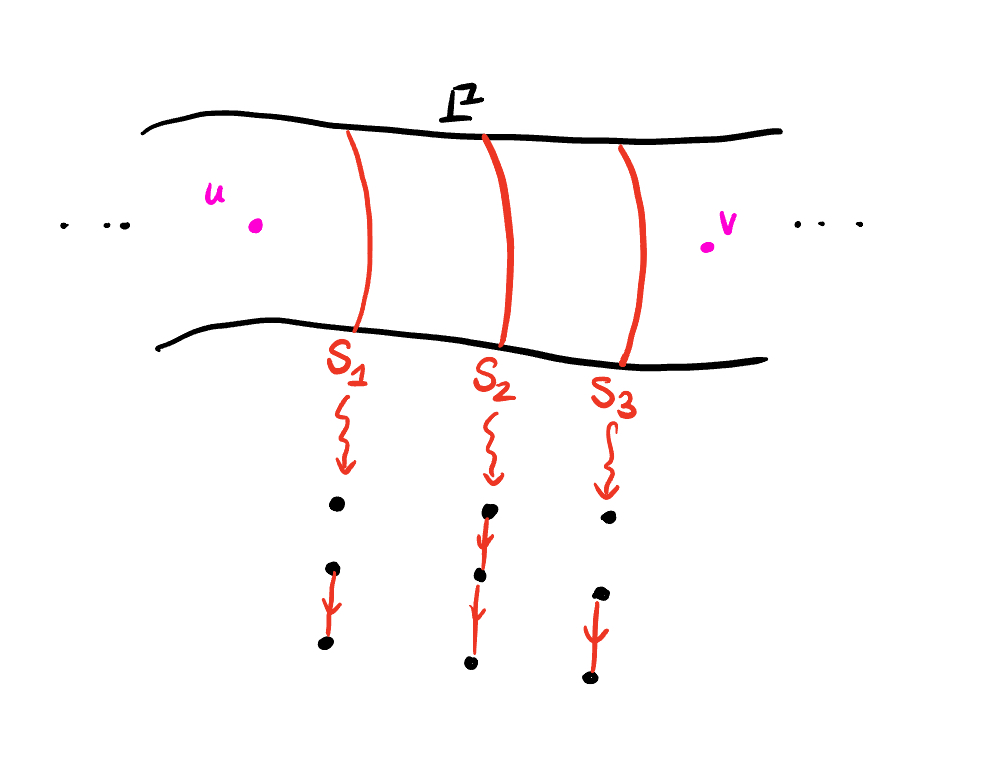}
    \caption{``Cross-Sections" for $S_i$. Each $S_i$  has at most $N$ vertices and the magnitude of the total flow is bounded by $N'$, so making the distance from $u$ and $v$ sufficiently large will guarantee that we will see the same graph (with the same flow) repeat.}
    \label{fig:Cross_Section}
\end{figure}

From this and the pigeonhole principle, one deduces that the directed graph $\Gamma$ must contain a pair of spheres $S,S'$ about $v$ with neighborhoods that are (directed-)graph isomorphic via an isomorphism $\phi$ that preserves the flow. Letting $[S,S']\subset \Gamma$ denote the portion of $\Gamma$ between $S$ and $S'$ (i.e., those vertices that are at least as far away from $v$ as is $S$ and at most as far away from $v$ as is $S'$), we consider a quotient $[S,S']/\sim$ obtained by identifying $S$ and $S'$. Crucially, because this identification can be made in a flow-preserving way, the flow we started with descends to the quotient. 

We now have a finite graph equipped with a map from its edge set to $\left\{0,1\right\}$ satisfying the properties of a (non-zero!) flow, which implies there must exist a cycle (a directed edge assigned with a $1$ must be followed by another edge assigned with a $1$-- repeat this argument until one arrives where one started). Lifting this cycle back to $\Gamma$ yields a path with one endpoint $x \in S$ and the other $y \in S'$, and if $g$ is an element sending $x$ to $y$, one can use the structure of the flow to show that $g$ must have infinite order. This is formalized with basic algebraic topology. Sweeping some of those details under the rug, the key point is that the cycle in $[S,S']/\sim$ corresponds to a (finite-index) subgroup\footnote{In truth, it is more straightforward to see that the connected component of $[S,S]/\sim$ containing the cycle -- or in fact any connected component of the quotient-- corresponds to a finite index subgroup. But in any case, the cycle will in turn be finite index in the subgroup corresponding to the component containing it.} of $G$ and $\langle g \rangle$ is evidently finite index in \textit{it} and therefore also in the full group; finally, $g$ can not be finite order because applying higher and higher powers of it corresponds to winding more and more times around the cycle.

\begin{figure}
    \centering
    \includegraphics[width=0.5\linewidth]{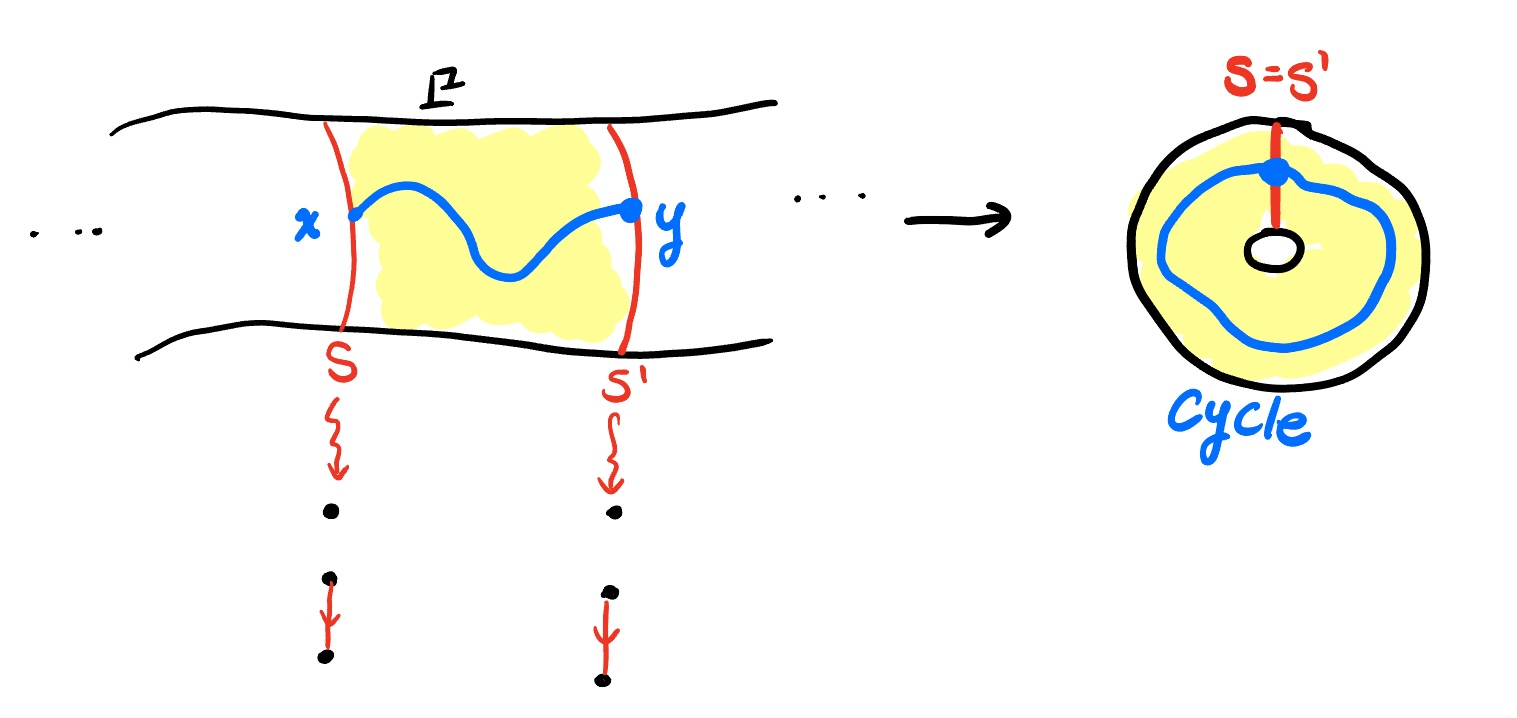}
    \caption{Quotient of the portion of $\Gamma$ between the two identical spheres $S,S'$ with the same flow.}
    \label{fig:Flow_Cycle}
\end{figure}

\section*{Appendix C: Flame's Math Art}

We would like to conclude this paper with some of Flame's own personal mathematical drawings. 

The drawings in Figures 12 and 13 come from Flame's senior thesis project at Haverford College, working under the supervision of the first author. Flame was studying the question of how many pairwise non-isotopic simple closed curves one can draw on the torus such that no two geometrically intersect more than $k$ times. For the interested reader, this problem was studied by the author and Gaster \cite{AG} and completely resolved by Balla-Filokovsk\'{y}-Kielak-Kr\'{a}l-Schlomberg \cite{BFKKS} and Kriepke-Schymura \cite{KS}. The figures in quadrants $I,II$, and $III$ were all in service of a proof of the familiar derterminant-based formula for computing the geometric intersection number between simple closed curves on the torus. The final figure was drawn when Flame was studying standard lifting lemmas in order to pass back and forth between curves and their preimages in the universal cover. The department has since established a prize in their memory and honor, awarded every year to a graduating senior whose thesis sports the most impressive mathematical figures. 

\begin{figure}[h]
    \centering
    \begin{minipage}{0.45\textwidth}
        \centering
        \includegraphics[width=0.9\linewidth]{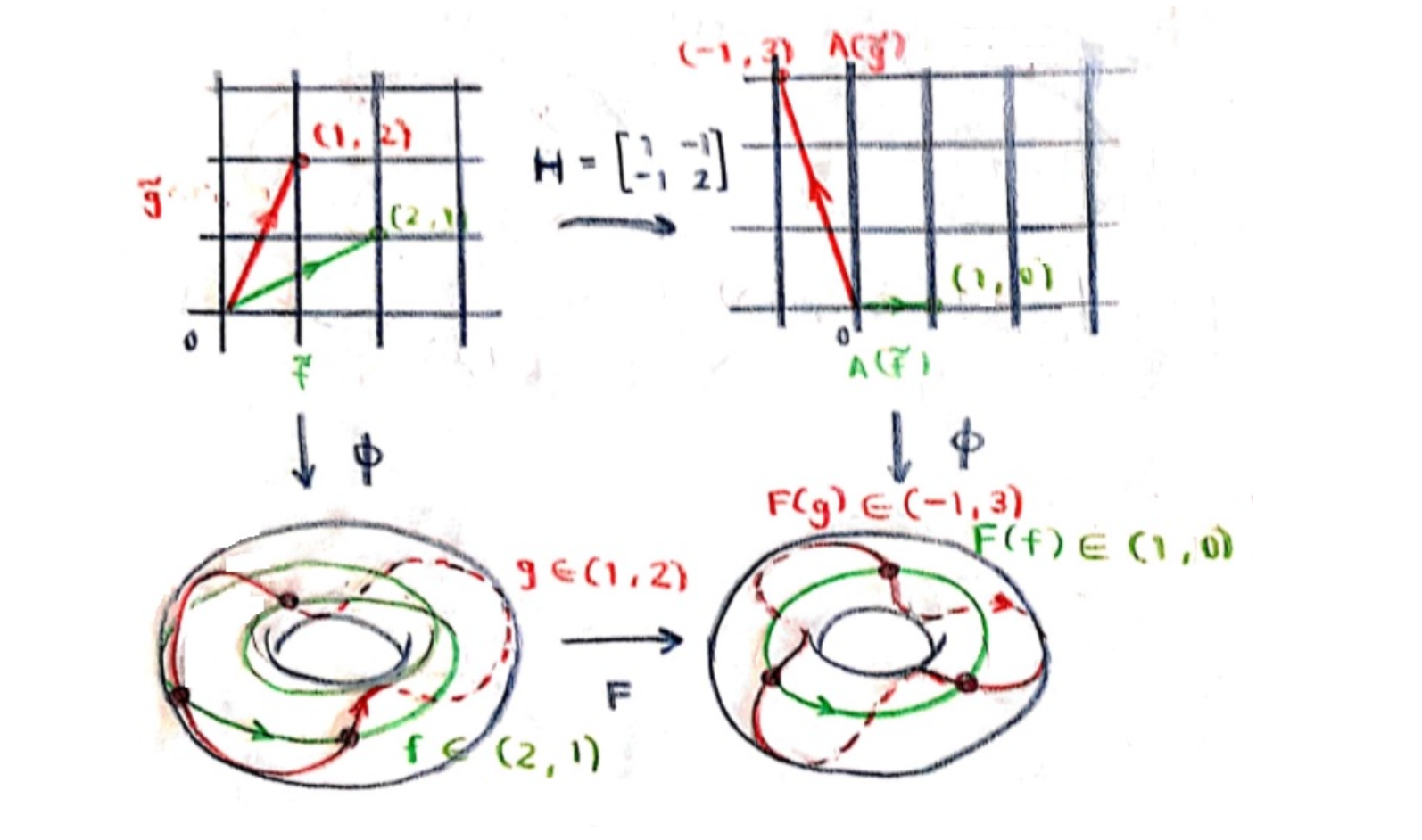}
    \end{minipage}\hfill
    \begin{minipage}{0.45\textwidth}
        \centering
        \includegraphics[width=0.9\linewidth]{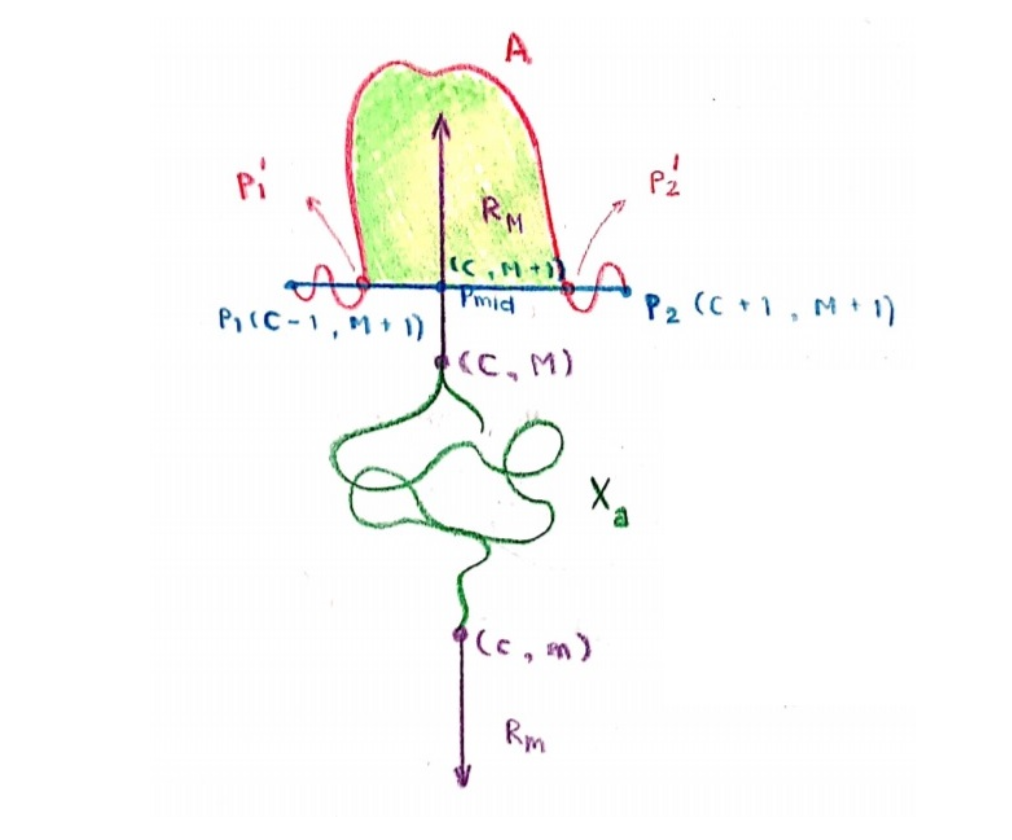}
    \end{minipage}
    \caption{}
    \label{fig:thesis}
\end{figure}

\begin{figure}[h]
    \centering
    \begin{minipage}{0.45\textwidth}
        \centering
        \includegraphics[width=0.9\linewidth]{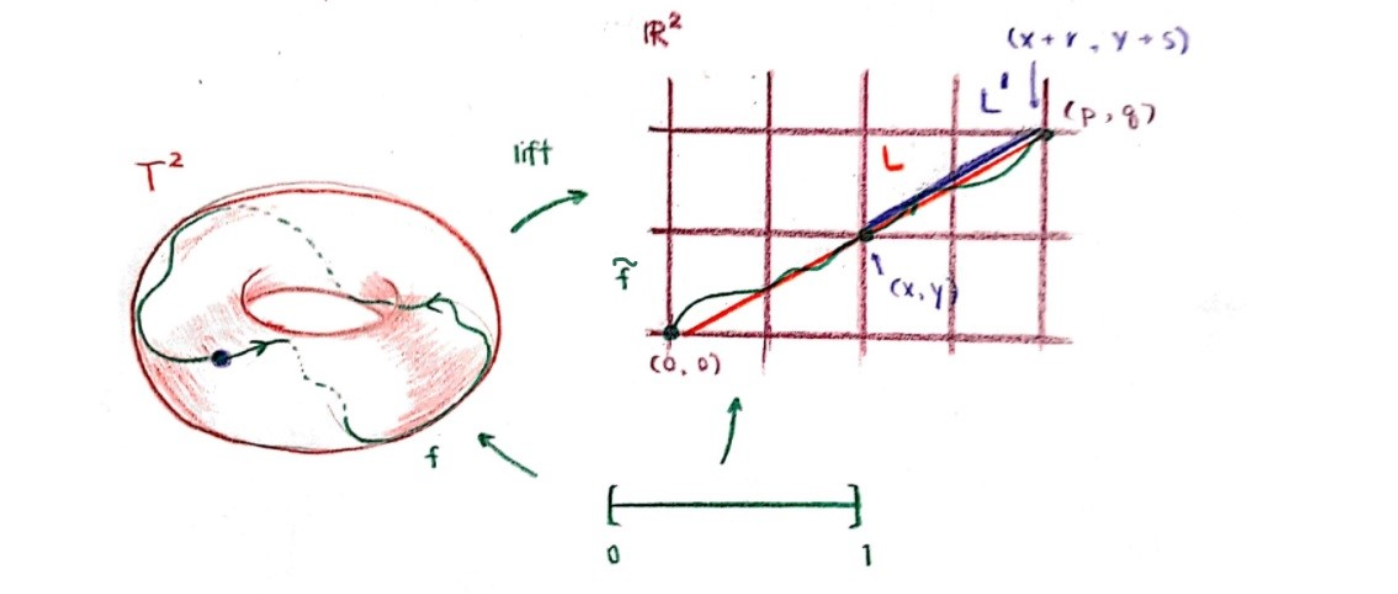}     
    \end{minipage}\hfill
    \begin{minipage}{0.45\textwidth}
        \centering
        \includegraphics[width=0.9\linewidth]{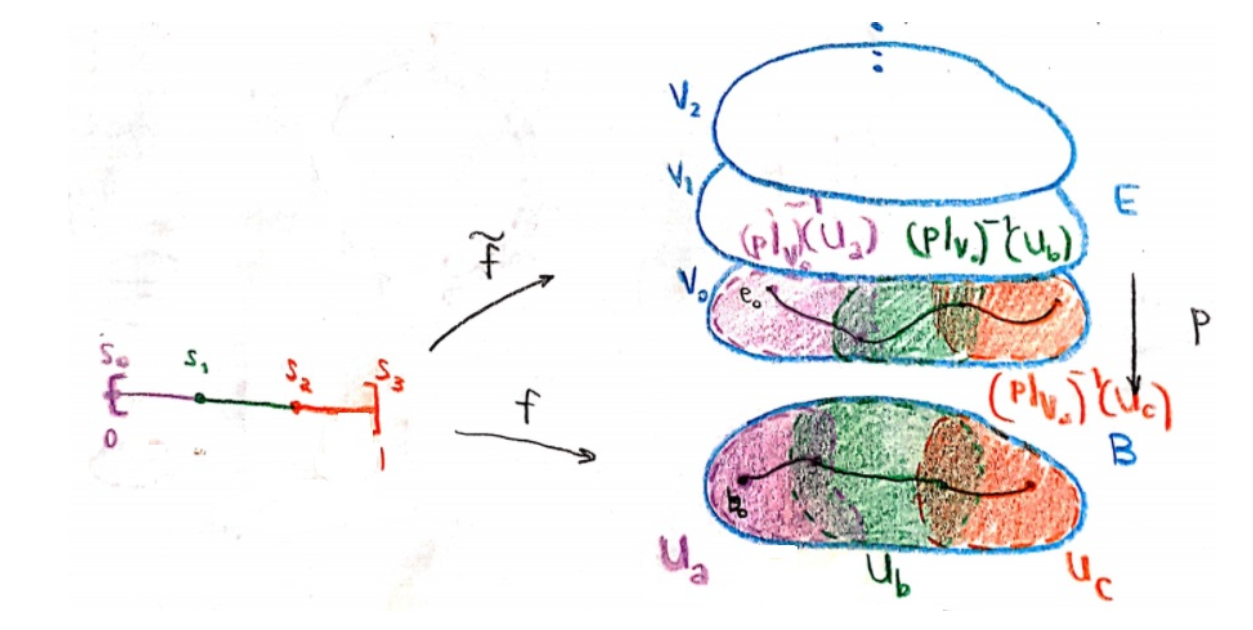}
    \end{minipage}
    \caption{}
    \label{fig:thesis}
\end{figure}

At the beginning of this article, we ask what the Cayley graph of the group $\langle a,b\ |\ ba^2b^{-1}=a\rangle$ looks like. In Figure \ref{fig:BS12}, we have part of the Cayley graph drawn on the left. In this graph, there are infinitely many planes branching off of each other in a tree-like fashion. The figure on the right is Flame's drawing of the geometry inside one of those planes. This picture illustrates the fact that the path along that bottom of length 8 that looks like a straight line is not the shortest path between the end points of that straight line. This implies that the geometry in each of the planes in the Cayley graph is more like the hyperbolic plane than the Euclidean plane. This observation gives us pertinent algebraic information about this group.

\begin{figure}[b]
    \centering
    \begin{minipage}{0.45\textwidth}
        \centering
    \includegraphics[scale=0.2,width=3cm, height=3cm]{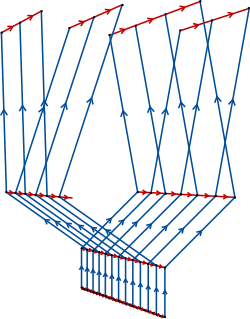}
    \end{minipage}
    \begin{minipage}{0.45\textwidth}
        \centering
        \includegraphics[width=\textwidth]{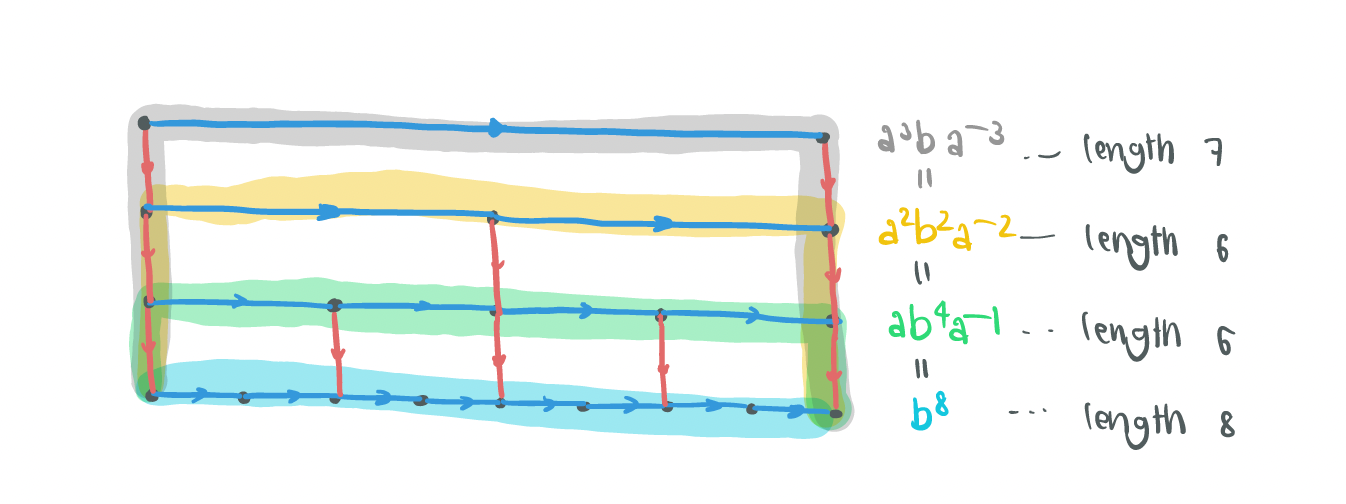}     
    \end{minipage}
    \caption{}
    \label{fig:BS12}
    \end{figure}

Figure \ref{fig:Pictures2} is a labeled copy of the Cayley graph for the group $G=\langle a,b\ |\ a^5=b^2=1\rangle$ drawn by Flame. The vertices are labeled with the corresponding group elements, the green edges correspond to multiplying by $b$ and the brown edges are multiplication by $a$. The group $G$ is the free product of 
$\mathbb Z_5$ and $\mathbb Z_2$. This is an example of a group with infinitely many ends. In fact, Stallings' theorem mentioned in Appendix A implies that any group with infinitely many ends has an algebraic splitting as a free product with amalgamation over a finite subgroup. A free product is a special case of this where the amalgamation is done over the trivial subgroup.

\begin{figure}[b]
    \centering
    \includegraphics[width=0.5\textwidth]{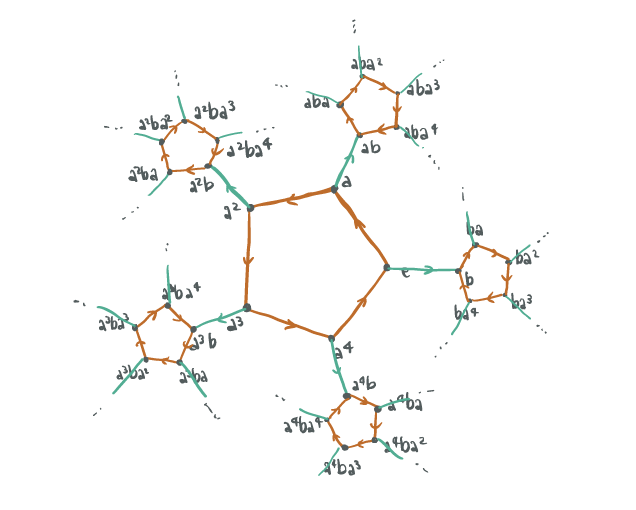}
    \caption{}
    \label{fig:Pictures2}
    \end{figure}

\newpage
\printbibliography
\vfill
\end{document}